\documentclass[11pt,reqno]{amsart}

\usepackage[margin=1.5cm]{geometry}
\usepackage{blindtext}
\usepackage{cite}
\usepackage{graphicx,geometry}
\usepackage{amssymb}
\usepackage{ulem}
\usepackage{color}
\usepackage{tikz}
\usepackage{multicol}
\usetikzlibrary{backgrounds}
\usetikzlibrary{arrows}
\usetikzlibrary{shapes,shapes.geometric,shapes.misc}

\tikzstyle{tikzfig}=[baseline=-0.25em,scale=0.5]

\pgfkeys{/tikz/tikzit fill/.initial=0}
\pgfkeys{/tikz/tikzit draw/.initial=0}
\pgfkeys{/tikz/tikzit shape/.initial=0}
\pgfkeys{/tikz/tikzit category/.initial=0}

\pgfdeclarelayer{edgelayer}
\pgfdeclarelayer{nodelayer}
\pgfsetlayers{background,edgelayer,nodelayer,main}

\tikzstyle{none}=[inner sep=0mm]

\newcommand{\tikzfig}[1]{%
{\tikzstyle{every picture}=[tikzfig]
\IfFileExists{#1.tikz}
  {\input{#1.tikz}}
  {%
    \IfFileExists{./figures/#1.tikz}
      {\input{./figures/#1.tikz}}
      {\tikz[baseline=-0.5em]{\node[draw=red,font=\color{red},fill=red!10!white] {\textit{#1}};}}%
  }}%
}

\tikzstyle{every loop}=[]

\tikzstyle{new style 0}=[fill=cyan, draw=black, shape=circle]

\usetikzlibrary{decorations.pathreplacing}
\usepackage{tikz-cd}
\usepackage{quiver}
\usepackage{lscape} 
\usepackage{mathrsfs}
\usepackage[all]{xy}
\numberwithin{equation}{section}
\newtheorem{theorem}{Theorem}[section]
\newtheorem{lemma}[theorem]{Lemma}

\newtheorem{corollary}[theorem]{Corollary}
\newtheorem{proposition}[theorem]{Proposition}

\theoremstyle{definition}
\newtheorem{example}[theorem]{Example}
\newtheorem{definition}[theorem]{Definition}
\newtheorem{definition-lemma}[theorem]{Definition-Lemma}
\newtheorem{definition-theorem}[theorem]{Definition-Theorem}
\newtheorem{remark}[theorem]{Remark}

\newcommand{\BC}{\mathbb{C}}
\newcommand{\BE}{\mathbb{E}}
\newcommand{\BF}{\mathbb{F}}
\newcommand{\BH}{\mathbb{H}}
\newcommand{\BL}{\mathbb{L}}
\newcommand{\CP}{\mathbb{P}}
\newcommand{\BR}{\mathbb{R}}
\newcommand{\BZ}{\mathbb{Z}}
\newcommand{\CC}{\mathcal{C}}
\newcommand{\TT}{\mathcal{T}}
\newcommand{\HH}{\mathcal{H}}
\newcommand{\CB}{\mathcal{B}}
\newcommand{\CD}{\mathcal{D}}
\newcommand{\CF}{\mathcal{F}}

\newcommand{\CO}{\mathcal{O}}

\newcommand{\CA}{\mathcal{A}}

\newcommand{\lan}{\langle}
\newcommand{\ran}{\rangle}

\newcommand{\wtil}{\widetilde}

\newcommand{\RHom}{\mathbb{R}\mathrm{Hom}}
\DeclareMathOperator{\Aut}{Aut}
\DeclareMathOperator{\HO}{H}
\DeclareMathOperator{\Hom}{Hom}
\DeclareMathOperator{\Ext}{Ext}
\DeclareMathOperator{\End}{End}
\DeclareMathOperator{\cone}{Cone}
\DeclareMathOperator{\sym}{Sym}
\DeclareMathOperator{\qcoh}{QCoh}

\DeclareMathOperator{\Coh}{Coh}
\DeclareMathOperator{\Rep}{Rep}
\DeclareMathOperator{\Sim}{Sim}
\DeclareMathOperator{\Stab}{Stab}

\DeclareMathOperator{\Tot}{Tot}
\allowdisplaybreaks

\begin{document}
\title{t-structures on a local 4-Calabi-Yau variety}

\address{University of Sheffield, the Hicks building, Hounsfield Road, Sheffield S3 7RH, UK}
\email{yxiong20@sheffield.ac.uk}

\author{Yirui Xiong}
\maketitle
\linespread{1.5}
\selectfont

\begin{abstract}
    Let $X= \Tot \Omega_{\CP^2}$ denote the total space of cotangent bundle of $\CP^2$. This is a non-compact Calabi-Yau $4$-fold (also called local Calabi-Yau variety in the physics literature). The aim of this paper is to use the tilting objects to characterize a large class of t-structures in the derived category of $X$ and study  the combinatorics of their simple tilts.
\end{abstract}

\section{Introduction}

The conception of the t-structure in the derived category (more generally triangulated category) first appeared in the paper of Beilinson, Bernstein and Deligne \cite{BBD} as a powerful tool to study the perverse sheaves of a stratified topological space. The key feature of a t-structure is giving a formal construction of abelian subcategories (called hearts) other than the standard one, and this becomes an important ingredient in the definition of stability conditions in triangulated categories by Bridgeland \cite{Bri07}. 

Let $\CD$ be a triangulated category, a stability condition on $\CD$ is a pair of $\sigma = (\CA, Z)$ where $\CA$  is a heart of some bounded t-structure in $\CD$ and a group homomorphism $Z: K(\CA) \rightarrow \BC$  satisfying Harder-Narasimhan property\cite[Proposition 5.3]{Bri07}. One of the most important properties is that the space of stability conditions $\Stab \CD$ (satisfying the local-finite conditions) will form a complex manifold, perhaps infinite dimensional \cite[Theorem 1.2]{Bri07}. In particular when the heart $\CA$ is of finite length, then we can get an open subset $U(\CA)\subset \Stab \CD$ which consists of stability conditions with fixed bounded hearts $\CA$. $U(\CA)$ is isomorphic to $\BH^n$ ($\BH$ is the upper half plane) where $n$ is the number of simple objects in $\CA$. And two such open subsets are glued along 1-codimensional subset if the corresponding hearts are related by a simple tilt (see Section \ref{t_str}).

The t-structures of $3$-Calabi-Yau categories have been studied a lot, and the space of stability conditions have been calculated in many cases\cite{Bri05, Bri06, KS08, BS15}. For example,  Bridgeland in \cite{Bri05} studied the case of total space of $\omega_{\CP^2}$ the canonical bundle where the t-structures are given by full exceptional collections on $\CP^2$, and simple tilts are related to mutations of exceptional collections on $\CP^2$.  In this paper, we  are focusing on $X = \Tot\Omega_{\CP^2}$ the total space of cotangent bundle, which is a non-compact Calabi-Yau $4$-fold. It is worth studying $X$ since it is the Springer resolution of the nilpotent orbit in $sl(3)$. The t-structures of the derived category of $X$ (more generally the cotagent bundle of flag varieties) were first studied in the work \cite{BZ06, BR12, BM13}, where the so-called exotic t-structures were found. In our paper, we constructed a large class of t-structures  by finding the tilting objects in $D^b(\Coh X)$. 

We firstly explain how to construct a t-structure by using tilting objects:

\subsection{}
Let $\CD = D^b(\Coh V)$ where $V$ is a variety over $\BC$, $T$ be an object in $\CD$ satisfying
\begin{enumerate}
    \item $\Hom_{\CD}\left(T,T[n]\right) = 0$, unless $n=0$;
    \item $T$ is  a classical generator of $\CD$, i.e., the smallest triangulated subcategory containing $T$ is $\CD$,
\end{enumerate}
then $T$ is called a tilting object. By Rickard's \cite{R89,Keller94} general theory of derived Morita equivalence, we have an equivalence between derived categories:
\begin{equation}\label{d_morita}
    \Hom_{\CD}\left(T, - \right): \CD \longrightarrow D^b(\text{mod-} B)
\end{equation}
where $B := \End_V T$. What we are mostly interested is that the above equivalence will give us a bounded t-structure $(\CD^{\leq 0}, \CD^{\geq 0})$ by pulling back the standard one on the right side:
\begin{eqnarray*}
    \CD^{\leq 0}&=& \left\{ F\in \CD: \Hom_{\CD}(T,F[i]) = 0,\ i>0 \right\},\\
    \CD^{\geq 0}&=& \left\{ F\in \CD: \Hom_{\CD}(T,F[i]) = 0,\ i<0 \right\}.
\end{eqnarray*}
We are characterizing the bounded t-structures of the above form in our case.

\subsection{}
We denote the projection map $\pi: X= \Tot \Omega_{\CP^2} \rightarrow \CP^2$. The usual strategy is to find tilting objects in $D^b(\Coh X)$ via pulling back the tilting objects from $\CP^2$ by $\pi$. More precisely, we consider the full exceptional collection $\BE = (E_0,E_1,E_2)$ (see definition \ref{exc_coll}) consisting of sheaves on $\CP^2$ thus $\oplus_i E_i$ is a tilting object on $\CP^2$, then we try to show that $\pi^*\left(\oplus_i E_i\right)$ is a tilting object for $X$.  However, this is not always possible. We need to use the helix theory: for any full exceptional collection $\BE = (E_0,\cdots, E_n)$, one can associate $\BH = (\cdots,E_{-1}, E_0,E_1,\cdots,E_n,\cdots)$ a collection of exceptional objects which is periodic under the Serre functor, and any successive $n+1$ objects in $\BH$ (called a thread) will be a full exceptional collection. Now our first result on tilting objects in $D^b(\Coh X)$ is  stated as the following:
\begin{proposition}[= Proposition \ref{tilting}]
	For any full and strong exceptional collection (see Definition \ref{exc_coll}) $\BE=(E_0,E_1,E_2)$ on $\CP^2$, one can always find a thread $(E_i,E_{i+1},E_{i+2})$ $i\in \BZ$ in the corresponding helix $\BH$, such that the pulling back $\bigoplus_{i}^{i+2} \pi^* E_i$ is a tilting object in $D^b(\Coh X)$.
\end{proposition}

Then we introduce a quiver associated with $\BE$ which encodes the non-tilting information of $\BE$ in the above, and is called the secondary quiver of $\BE$. We denote it by $Q_{\BE}$. After analyzing possible types of the secondary quivers (which are all acyclic), we define a functor $\CF_{Q}$ from mod-$\BC Q_{\BE}$ to $\Coh X$, and prove the following theorem, which gives us additional tilting objects associated with $\BE$:
\begin{theorem}[=Theorem \ref{F_Q}]
    For any full and strong exceptional collection $\BE$ on $\CP^2$, $\CF_{Q}$ sends tilting objects in mod-$\BC Q_{\BE}$ to tilting objects in $D^b(\Coh X)$.
\end{theorem}
Here tilting objects in mod-$\BC Q_{\BE}$ means they are tilting objects when viewed as objects in $D^b(\BC Q_{\BE})$. 

Let $D_0^b(\Coh X)$ be the full subcategory of $D^b(\Coh X)$ consisting of objects with support on the zero section $\CP^2\subset X$, the equivalence (\ref{d_morita}) above determines a t-structure on $D_0^b(\Coh X)$ whose heart $\CB\subset D_0^b(\Coh X)$ is an abelian category which is equivalent to the category of nilpotent modules over $B  = \End_X T$. In the final chapter, we calculate simple tilts from a given one. The calculations involve an interesting autoequivalence on $D^b(\Coh X)$ which is induced by the Mukai flop (see Proposition \ref{swap}), first studied by Namikawa \cite{Nam03} and Kawamata \cite{Kaw02}.

We remark that the combinatorics of simple tilts are complicated, and not controlled entirely by the mutations of exceptional collections.

\section*{Notation}
For a complex variety $V$, we write $D^b(V) := D^b(\Coh V)$ the bounded derived category of coherent sheaves on $V$. By sheaf on $V$ we mean coherent $\CO_V$-module. We write $D^b(A)$ for the bounded derived category of right $A$-modules over a notherian $\BC$-algebra $A$. We write $\CO(i)$, $\Omega$, $\TT$ the Serre twisted, cotangent and tangent sheaves on $\CP^2$.

\section*{Acknowledgements}
This work forms part of the author's thesis. He would like to thank his supervisor Tom Bridgeland for  the guidance and generosity on sharing the insights on this topic, as well as his advisor Evgeny Shinder for his continued supports. He would like to thank Nebojsa Pavic and the members in Sheffield algebraic geometry and mathematical physics group for the helpful discussions. The project is sponsored by the China Scholarship Council (CSC) and the University of Sheffield.

\section{t-structures, Helices and Mutations}
In this chapter we assume $\CD$ is a triangulated category such that 
\begin{enumerate}
	\item $\CD$ is a $\BC$-linear category: the $\Hom$ spaces are $\BC$-linear spaces.
	\item $\CD$ is of finite type, i.e. for any two objects $A$, $B$ of $\CD$ the vector space
	$$
		\bigoplus_{i\in \mathbb{Z}}\Hom^i_{D} (A,B)
	$$
	is finite-dimensional.
	\item $\CD$ is algebraic in the sense of Keller\cite{Keller94}.
	\item $\CD$ is saturated \cite{BK}, i.e., all triangulated functors
	$$
	    \CD \longrightarrow D(\BC) , \quad \CD^{op}\longrightarrow D(\BC),
	$$
	are representable.
\end{enumerate}

\subsection{}\label{t_str}
Let $(D^{\leq 0}, D^{\geq 0})$ be a bounded t-structure on $D$. Any such t-structure is determined by its heart  $D^{\heartsuit} = D^{\leq 0}\cap D^{\geq 0}$, which is a full abelian subcategory. For a bounded heart $\CA$ we shall use  $\HO^i_{\CA}(-)$ to indicate the cohomological functor with respect to the heart. A heart will be called finite-length if it is artinian and noetherian as an abelian category.

The following definition comes from Happel-Reiten-Smalo \cite{HRS}.
\begin{definition}[Torsion pair]
Let $\CA$ be a heart of some t-structure in the triangulated category $\CD$. A pair of full subcategories $(\TT,\CF)$ of $\CA$ is called a torsion pair in $\CA$ if it satisfies the following conditions
   \begin{enumerate}
       \item $\Hom_{\CA}(T,F) = 0$ for $T\in \TT$ and $F\in \CF$;
       \item for any object $A\in \CA$, there exist $M\in \TT$ and $N\in \CF$ such that they fit into a short exact sequence
       $$
        \xymatrix{
         0 \ar[r] & M \ar[r] & A \ar[r] & N \ar[r] & 0.
        }
       $$
   \end{enumerate}
\end{definition}

 The following theorem was proved in \cite[Proposition 2.1]{HRS}.
 \begin{theorem}[Happel-Reiten-Smalo]
Let $(\TT,\CF)$ be a torsion pair in a heart $\CA$. Let
    \begin{eqnarray*}
        \CA^{\sharp} &:=& \left\{ E\in D\ | \ \HO^1_{\CA}(E)\in \TT,\ \HO^0_{\CA}(E)\in \CF, \ \HO^i_{\CA}(E) = 0 \text{ for }i\neq0,\ 1 \right\}, \\
        \CA^{\flat} &:=& \left\{ E\in D \ | \ \HO^{-1}_{\CA}(E) \in \CF,\ \HO^0_{\CA}(E)\in \TT,\ \HO^i_{\CA}(E) = 0 \text{ for }i\neq-1,\ 0 \right\},
    \end{eqnarray*}
    then both $\CA^{\sharp}$ and $\CA^{\flat}$ are hearts of some new bounded t-structures in $\CD$.
 \end{theorem}
 A special case of the tilting construction will be particularly important. Suppose that $\CA$ is a finite-length heart and $S\in \CA$ is a simple object. Let $\langle S\rangle$ be the full subcategory consisting of objects $E\in \CA$ all of whose simple factors are isomorphic to $S$. Define the full subcategories 
 $$
    S^{\perp} := \{E\in \CA \ | \ \Hom_{\CA}(S, E) = 0\}, \quad {}^{\perp}S:= \{E\in\CA \ |\ \Hom_{\CA}(E, S) = 0\}.
 $$
 Then we can either view $(\langle S \rangle, S^{\perp})$  or $(^{\perp} S, \langle S \rangle )$ as a torsion pair. Then we can define new tilted hearts
 \begin{equation} \label{rltilt}
     R_S \CA := \langle S[1], ^{\perp} S \rangle, \quad L_S \CA := \langle S^{\perp}, S[-1] \rangle,
 \end{equation}
 which we refer to as the right and left tilts of the heart $\CA$ at the simple $S$. It is easy to see that $S[-1]$ is a simple object of $L_S \CA$ and that if the category if of finite-length, then $R_{S[-1]} L_S \CA = \CA$. Similarly, if $R_S \CA$ is of finite-length then $L_{S[1]} R_S \CA = \CA$, see Figure \ref{fig:mut-hearts}.

 \begin{figure}
\begin{center}
\begin{tikzpicture}[scale=0.4]
\draw (0,4)--(13,4);
\draw (0,9)--(13,9);
\draw (1.5,4)--(1.5,9);
\draw (4,4)--(4,9);
\draw (9.25,4)--(9.25,9);
\draw (11.75,4)--(11.75,9);
\draw (2.75,6.5) node {$S[1]$};
\draw (6.5,6.5) node {$^\perp\!S$};
\draw (10.5,6.5) node {$S$};
\draw [decorate, decoration={brace,amplitude=5pt}] (4,9.5)--(11.75,9.5)
node [midway, above=6pt] {$\CA$};
\draw [decorate, decoration={brace,amplitude=5pt}] (9.25,3.5)--(1.5, 3.5)
node [midway, below=6pt] {$R_S\CA$};
\end{tikzpicture}
\quad \quad
\begin{tikzpicture}[scale=0.4]
\draw (0,4)--(13,4);
\draw (0,9)--(13,9);
\draw (1.5,4)--(1.5,9);
\draw (4,4)--(4,9);
\draw (9.25,4)--(9.25,9);
\draw (11.75,4)--(11.75,9);
\draw (2.75,6.5) node {$S$};
\draw (6.5,6.5) node {$S^\perp$};
\draw (10.5,6.5) node {$S[-1]$};
\draw [decorate, decoration={brace,amplitude=5pt}] (1.5,9.5)--(9.25,9.5)
node [midway, above=6pt] {$\CA$};
\draw [decorate, decoration={brace,amplitude=5pt}] (11.75,3.5)--(4, 3.5)
node [midway, below=6pt] {$L_S \CA$};
\end{tikzpicture}
\caption{Left and right tilts of a heart.}\label{fig:mut-hearts}
\end{center}
\end{figure}
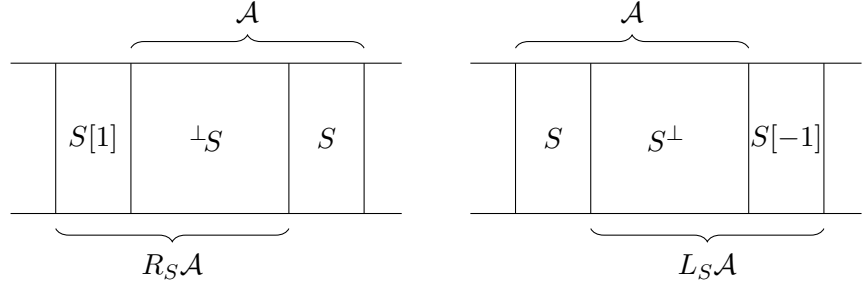

 The following lemma will be useful in Section\ref{cal_tilt}.
 \begin{lemma}\label{aut_sim_tilt}
    The operation of tilting commutes with the action of the group of autoequivalences on the set of t-structures: take an autoequivalence $\Phi\in\Aut (\CD)$. Let $\CA\in \CD$ be a heart of some t-structure and of finite-length, $S\in \CA$ be a simple object. Then we have 
    $$
        \Phi(L_S\CA) = L_{\Phi(S)} \Phi(\CA).
    $$
 \end{lemma}
\begin{proof}
By (\ref{rltilt}), we have $\Phi(L_S\CA) = \left\langle \left(\Phi(S)\right)^{\perp}, \Phi(S)[-1]\right\rangle = L_{\Phi(S)}\Phi(\CA)$.
\end{proof}
Given a heart of bounded t-structure $\CA\in \CD$, we denote by $\Sim \CA$ the complete set of non-isomorphic simple objects in $\CA$.
\begin{proposition}\label{sim_tilt}
    Assume $\Sim \CA$ is finite and $\CA$ is finite-length. Let $S\in \CA$ such that $\Ext_{\CA}^1(S,S) = 0$, then after a left or right simple tilting, the new simple objects are:
    \begin{eqnarray}
        \Sim L_S\CA &=& \{ S[-1]\} \ \cup \ \{\phi_S(X): X\in \Sim \CA,\ X\neq S\} \\
        \Sim R_S\CA &=& \{S[1] \} \ \cup \ \{\psi_S(X): X\in \Sim \CA,\ X\neq S\}
    \end{eqnarray}
    where 
    \begin{eqnarray*}
        \phi_S(X) &=& \cone\ \left(S[-1]\otimes \Ext^1(S,X) \longrightarrow X\right),\\
        \psi_S(X) &=& \cone\ \left(X\longrightarrow S[1]\otimes \Ext^1(X,S)^*\right)[-1].
    \end{eqnarray*}
\end{proposition}
\begin{proof}
    We refer our reader to \cite[Proposition 5.4]{KQ15} for the proof.
\end{proof}

\subsection{Exceptional collections}
\begin{definition}[Exceptional collections] \label{exc_coll}
An object $E$ in $\CD$ is said to be exceptional if 
$$
	\Hom^k_{\CD}(E,E) = \left\{ \begin{array}{cc}
								\BC,&\  \text{if } k=0,\\
								0, &\  \text{otherwise}
								\end{array}
							\right.
$$
An exceptional collection ${\BE}\subset \CD$ is a sequence of exceptional objects 
$$
	{\BE}=(E_0,\cdots,E_n)
$$
such that for all $0\leq i<j\leq n$, we have $\Hom^{\bullet}_{\CD}(E_j,E_i) = 0$.
\end{definition}
An exceptional collection $\BE = (E_0,\cdots, E_n)$ is said to be strong if for all $i,\ j$ 
$$
	\Hom^k_{\CD}(E_i,E_j) = 0,\ \ \text{unless } k=0.
$$

We write $\langle\BE\rangle \subset \CD$ for the smallest full triangulated subcategory of $\CD$ containing the elements of an exceptional collection $\BE\subset \CD$. An exceptional  collection $\BE$ is said to be full if $\langle \BE \rangle = \CD$.

From the definitions above, we have that for a full and strong exceptional collection $\BE$, then $\bigoplus_{i=0}^{n} E_i$ is a tilting object in $\CD$.
Given an exceptional collection $\BE$ in  $\CD$, the right orthogonal subcategory to $\BE$ is the full triangulated subcategory 
$$
	\BE^{\bot}=\left\{ X\in \CD: \ \Hom^{\bullet}_{\CD}(E,X) = 0\ \text{for } E\in \BE\right\}.
$$
Similarly, the left orthogonal subcategory to $\BE$ is 
$$
	{}^{\bot} \BE = \left\{ X\in \CD: \ \Hom^{\bullet}_{\CD}(X,E) = 0\ \text{for } E\in \BE\right\}.
$$
 $\langle \BE\rangle$ is admissible due to \cite[Theorem 3.2]{B89}, i.e. the inclusion functor $i: \langle \BE \rangle \rightarrow \CD$ has left and right adjoint functors. Thus the fullness of $\BE$ is equivalent to $\BE^{\bot} = 0$ or ${}^{\bot} \BE = 0$.

\subsection{Mutation of exceptional collection}
We suppose $E\in \CD$ to be exceptional. Given an object $X\in {}^{\bot}E$, the left mutation of $X$ through $E$ is the object $L_E(X)\in E^{\bot}$ defined up to isomorphism by the triangle
$$
	\xymatrix{
		L_E(X) \ar[r]& \Hom^{\bullet}_D(E,X)\otimes E \ar[r]^-{ev} & X \ar[r]&L_E(X)[1],
		}
$$
where $ev$ denotes the evaluation map. Similarly, given $X\in E^{\bot}$, the right mutation of $X$ through $E$ is the object $R_E X \in {}^{\bot} E$ defined by the triangle
$$
	\xymatrix{
     X \ar[r]^-{coev} & \Hom^{\bullet}_D(X,E)^* \otimes E \ar[r] &	R_E(X) \ar[r] &  X[1],
	}
$$
where $coev$ denotes the coevaluation map. Then these two operations define mutually inverse equivalences of categories (see \cite[Appendix B]{BriS10})
$$
	\xymatrix{
		{}^{\bot}E \ar@/^/[rr]^{L_E} & & E^{\bot} \ar@/^/[ll]^{R_E}
	}
$$

\begin{definition-lemma}[Dual objects] \label{dual_coll}
Let $\BE = (E_0,\cdots, E_n)$ be a full exceptional  collection and define 
$$
	F_j = L_{E_0} \cdots L_{E_{j-1}} (E_j)[j],\ \ 0\leq j \leq n.
$$
Then $F_j$ is called dual object to $E_j$ and satisfies the following property:
$$
	\Hom^k_{\CD}(E_i,F_j) = \left\{ \begin{array}{cc}
							\BC & \text{if}\ i=j \ \text{and} \ k=0,\\
							0 & \text{otherwise.}
						\end{array}
						\right.
$$
\end{definition-lemma}
\begin{proof}
One can directly compute $R_{E_j} E_i =  E_i$ for $j<i$ and $R_{E_i} E_i = 0$. Then the $\Hom$ space when $j\leq i$ will be 
\begin{eqnarray*}
	\Hom^k_{\CD}(E_i,F_j) &=& \Hom^k_{\CD}\left(E_i, L_{E_0} \cdots L_{E_{j-1}} (E_j)[j]\right)\\
					&=& \Hom^{k-j}_{\CD}\left(R_{E_{j-1}} \cdots R_{E_0} E_i, E_j\right) \\ 
					&=& \Hom^{k-j}_{\CD}(E_i, E_j)\\
					&=&\left\{ \begin{array}{cc}
							\BC & i=j \ \text{and} \ k=0,\\
							0  & \text{otherwise.}
						\end{array}
						\right.
\end{eqnarray*}
When $j>i$ we have $\Hom^k_{\CD}(E_i,F_j) = 0$ for any $k$ from the above calculations.

\end{proof}

\begin{definition}[Standard mutation]
Given a full exceptional collection $\BE = (E_0,\cdots, E_n)$, the mutation operation $\sigma_i$ for each $0<i\leq n$ is defined by the rule
\begin{eqnarray*}
	&&\sigma_i(E_0,\cdots, E_{i-2}, E_{i-1},E_i, E_{i+1},\cdots, E_n) \\
	&=&(E_0,\cdots, E_{i-2}, L_{E_{i-1}}(E_i),E_{i-1}, E_{i+1},\cdots, E_n)
\end{eqnarray*}
Then $\sigma_i$ take exceptional collections to exceptional collections. And it takes full collections to full collections.
\end{definition}

\subsection{Mutation of Helices}
Under our assumptions for category $\CD$, according to  \cite[Corollary 3.5]{BK} $\CD$ has a Serre functor $S_D$. By the definition of the Serre functor, we have that $S_D$ is an autoequivalence of $\CD$ and there are given bi-functorial isomorphisms
$$  
    \phi_{E,G}: \Hom_{\CD}(E,G) \longrightarrow \Hom_{\CD}\left(G,S_{\CD}(G)\right)^* 
$$
for $E,\ G \in \CD$, with the following property: the composition 
$$
(\phi^{-1}_{F(E),F(G)})^*\circ \phi_{E,G}:\Hom(E,G)\rightarrow \Hom\left(G,S_{\CD}(E)\right)^* \rightarrow \Hom_{\CD}\left(S_{\CD}(E), S_{\CD}(F)\right)
$$
coincides with the isomorphism induced by $S_{\CD}$. A motivating example of the Serre functor is the following: when $\CD = D^b(Y)$ for smooth projective variety $Y$, then\\ $S_{\CD} := -\otimes \omega_Y[-\dim Y]$.
\begin{definition}[Helix]
A sequence of objects $\BH = (E_i)_{i\in \BZ}$ in $\CD$ is a helix if there exist positive integers $(n,d)$ such that 
\begin{enumerate}
	\item for each $i\in \BZ$ the corresponding thread $(E_{i},\cdots, E_{i+n})$ is a full exceptional collection,
	\item for each $i\in \BZ$ one has $E_{i-n-1} = S_{\CD}(E_i) [-d]$.
\end{enumerate}
\end{definition}

Given a full exceptional $\BE$, we can generate a helix using the Serre functor $S_{\CD}$ in the way by defining $E_{i-m(n+1)} := (S_{\CD})^m (E_i)[-dm]$ for integer $m$. Similarly we have  the mutation for helices: given a helix $\BH = (E_j)_{j\in \BZ}$ and an integer $i$ modulo $n+1$ , then we can get a new helix $\sigma_i(\BH) = \BH' = (E'_j)_{j\in\BZ}$:
$$
	E'_j = \left\{ 
		\begin{array}{cc}
			E_{j-1} & \text{if } j=i \text{ mod } n+1\\
			L_{E_j}(E_{j+1}) & \text{if } j=i-1 \text{ mod } n+1\\
			E_j & \text{otherwise}
		\end{array}
		\right.
$$
We want to ask which property will be preserved under mutation, for example, strong exceptional collection will not be preserved under mutation.
\begin{definition}[Geometric helix]
	A helix $\BH=(E_i)_{i\in\BZ}$ is called geometric if for any $i<j$ we have
	$$
		\Hom^k_{\CD}(E_i,E_j) = 0,\ \ k\neq 0.
	$$
	Equivalently, for each thread $(E_i,\cdots, E_{i+n})$ we have
	$$
		\Hom^k_{\CD}\left(E_i, S^{-l}_{\CD}(E_j)[ld]\right) = 0, \quad \text{for } k\neq 0, \ l\geq 0,\ i<j
	$$
\end{definition}
Given a $t$-structure $(\CD^{\leq 0},\CD^{\geq 0})$, a helix $\BH=(E_i)_{i\in\BZ}$ is called pure if each element $E_i$ is contained in the heart $\CD^{\heartsuit}$.

We have the following important results \cite[Theorem 2.3 and Lemma 2.5]{BP}.
\begin{theorem}[Bondal-Polishchuk] \label{mut_pres}
The pureness and geometricity of a helix of type $(n,n)$ are preserved under mutations.
\end{theorem}

Finally we give some useful properties of exceptional objects in $D^b(\CP^2)$.

Here we recall the conception of (semi)stability of sheaves on $\CP^n$: fix an ample divisor on $\CP^n$, here we can choose the hyperplane $H$, then the slope of a pure sheaf (see \cite[Definition 1.1.2]{HL10}) $\CF$ is defined to be 
$$
	\mu(F) := \frac{c_1(\CF).H^{n-1}}{r(\CF)},
$$
where $r(\CF)$ is the rank of $\CF$.
\begin{definition}[(Semi)stability]
A pure sheaf $\CF$ on $\CP^n$ is called stable (resp. semi-stable), if for any subsheaf $\mathcal{E}$ with strictly smaller rank than $\CF$, then we have 
$$
	\mu(\mathcal{E}) < \mu(\CF) \ \text{(resp. } \mu(\mathcal{E}) \leq \mu(\CF) \text{)}
$$
\end{definition}
The following useful lemma is proved in\cite[Theorem 4.1]{GR87} and \cite[Proposition 2.9]{KO94}.
\begin{lemma}\label{stab_free}
	If $E$ is an exceptional object on $\CP^2$, then $E$ is a shift of a locally free and stable sheaf.
\end{lemma}
A good proof of the following result is given by Bondal and Polishchuk \cite[Example 3.2]{BP}.
\begin{lemma}[Gorodentsev, Rudakov] \label{markov}
If $(E_0,E_1,E_2)$ is a full and strong exceptional collection in $D^b(\CP^2)$, then the positive integers $(a,b,c)$ defined by $a= \dim \Hom_{D^b(\CP^2)}(E_0,E_1)$, $b = \dim \Hom_{D^b(\CP^2)} (E_1,E_2)$ and $c=\dim \Hom_{D^b(\CP^2)}(E_0,E_2)$ satisfy the Markov equation:
$$
	a^2+b^2+c^2 = abc.
$$
\end{lemma}
Using the above lemmas, we can compare the slopes between bundles in an exceptional collection in $D^b(\CP^2)$:
\begin{lemma}\label{inc}
Let $\BE=(E_0,E_1,E_2)$ be a full and strong exceptional collection consisting of sheaves in  $D^b(\CP^2)$, then for $j>i$, we have $\mu(E_j) > \mu(E_i)$.
\end{lemma}
\begin{proof}
	If not, we suppose $\mu(E_j)\leq \mu(E_i)$ for some $j > i$. Then we claim that $\Hom_{\CP^2}(E_i,E_j) = 0$ since
	\begin{enumerate}
		\item if $\mu(E_j) = \mu(E_i)$. Any non-trivial homomorphism between stable bundles of equal slopes must be an isomorphism. However, since $\Hom_{\CP^2}(E_j,E_i) = 0$  by definition of an exceptional sequence, then we could only have $\Hom(E_{i}, E_j) = 0$;
		\item if $\mu(E_j) < \mu(E_i)$, then $\Hom_{\CP^2}(E_i, E_{j}) = 0$ since there is no map between stable bundles from the one with bigger slope to the smaller slope.
	\end{enumerate}
 Then by Lemma \ref{markov}, we have $a= b = c = 0$. Thus there is no map among $E_0$, $E_1$ and $E_2$.
	
	Since the exceptional collection is full, we see that $D^b(\CP^2)$ can be decomposed into two subcategories $\langle     E_0\ran$ and $\lan E_1,\ E_2\ran$. But by \cite[Example 3.2]{Bri99}, this would imply that $\CP^2$ is not connected, which is absurd.
\end{proof}
The following known result will be useful to us:
\begin{theorem} \label{strong}
Let $\BE = (E_0,E_1,E_2)$ be a full exceptional collection which consists of sheaves in $D^b(\CP^2)$. Then $\BE$ is strong.
\end{theorem}
\begin{proof}
By \cite[Corollary 2.11]{KO94}, for $i<j$ there is at most one $\Ext^k_{\CP^2}(E_i, E_j)$ group which doesn't vanish, and such $k \neq 2$. So either $\Hom_{\CP^2}(E_i,E_j) \neq 0$ or $\Ext^1_{\CP^2}(E_i,E_j) \neq 0$. By using Riemann-Roch theorem for exceptional sheaves, we have
$$
    \chi(E_i,E_j) = \frac{1}{2}r(E_i)r(E_j) \left( (\mu(E_i)-\mu(E_j) )^2+ 3(\mu(E_j)-\mu(E_i))+\frac{1}{r^2(E_i)}+\frac{1}{r^2(E_j)} \right).
$$
By Lemma \ref{inc}, $\mu(E_j)> \mu(E_i)$ when $j>i$. Thus $\chi(E_i,E_j) >0$, which implies $\Ext^1_{\CP^2}(E_i,E_j) = 0$.
\end{proof}

\section{Proof of Main theorem I}
First we show for an object which classically generates the derived category of $\CP^2$, then the pull back of such object under bundle projection map $\pi:X = \Tot \Omega_{\CP^2} \rightarrow \CP^2$ will classically generate $D^b(X)$. We say $T\in D^b(X) $ is a generator in the sense of \cite[Section 2.1]{BVdB}, if $\Hom^{\bullet}_D(T, A) = 0$ implies $A = 0$ in $D^b(X)$.
\begin{lemma}\label{gen}
Let $E$ be an object which classically generates $D^b(\CP^2)$, then $\pi^* E$ classically generates $D^b(X)$.
\end{lemma}
\begin{proof}
	 Let $\qcoh (X)$ be the category of quasi-coherent sheaves on $X$. By \cite[Theorems 2.1.2, Theorem 3.1.1]{BVdB}, it is equivalent to show that $\pi^*E$ generates the category $D\left(\qcoh (X)\right)$. Since $\pi :X \rightarrow \CP^2$ is affine and flat, then (non-derived) $\pi_*$ and $\pi^*$ are both exact. Thus we have $\BL \pi^* = \pi^* $ and $\BR \pi_* = \pi_*$. 
	 
	 For any object $F\in D\left(\qcoh X\right)$, we have the adjunction
	\begin{equation} \label{adj}
		\Hom^{\bullet}_{D\left(\qcoh X\right)}\left(\pi^*E, F\right) \cong \Hom^{\bullet}_{D\left(\qcoh \CP^2\right)} \left(E, \pi_*F\right).
	\end{equation}
	Then for any $F\in D(\qcoh X)$ suppose we have $\Hom^{\bullet}_{D(\qcoh X)}\left(\pi^*E, F\right) = 0$, then due to adjunction (\ref{adj})  and the fact that $E$ generates  $D(\qcoh(\CP^2))$, we have $\pi_*F = 0$. Since $\pi_*$ has no kernel (this can be verified by taking an affine cover on $\CP^2$ then applying the definition of $\pi_*$), we have $F\cong 0$. 
	
	Applying the same argument in reverse this means that $\pi^*E$ classically generates $D^b(X)$.
\end{proof}

We need the following results on vanishing of some cohomology groups. 
\begin{lemma} \label{vanish_4}
Suppose $\BE =(E_0,E_1,E_2)$ is a full exceptional sequence consisting of sheaves in $D^b(\CP^2)$ such that for any $i,\ j \in \{0,\ 1,\ 2\}$ we have $\mu(E_i)-\mu(E_j) \leq 2$. For any object $F\in D^b(\CP^2)$, we denote by $F|_H$ the restriction of $F$ to the hyperplane $H\subset \CP^2$.

Then we have for $n>0$
\begin{equation}\label{vanish_0}
	\HO^1\left(\CP^1, E^*_i \otimes E_j(n) |_H \right) = 0,
\end{equation}

\end{lemma}
\begin{proof}
Let $i,\ j\in \{0, 1, 2\}$. By assumptions and Lemma \ref{inc} we know $\mu\left(E_i(-2)\right) = \mu(E_i) - 2 \leq \mu(E_j)$. And since $E_j$ and $E_i(-2)$ are stable, we have $\Hom\left(E_j, E_i(-2) \right) = 0$  or\\ $\Hom\left(E_j, E_i(-2) \right) = \BC$.

Using the Serre duality we have 
$$
	\Ext^2_{\CP^2}\left(E_i, E_j(-1)\right) = \Hom_{\CP^2}\left(E_j, E_i(-2)\right)^*.
$$
Thus we have $\Ext^2\left(E_i, E_j(-1)\right) = \Hom\left(E_j, E_i(-2)\right)^* = 0$ or $\BC$.

 Let $H = \CP^1$ be a hyperplane in $\CP^2$, then by the Grothendieck splitting theorem, we denote the restriction of bundle $E_i$ to $H$ by $E_i|_H =\bigoplus_{m\in I} \CO(s_m)^{k_m}$, where $I$ is a finite set of indices, $s_m \in \BZ$, $k_m\in \mathbb{N}^+$. We denote by $E_i^{*}\otimes E_j|_H = \bigoplus_{m\in I'} \CO(a_m)^{u_m}$ for $a_m\in \BZ$ and $u_m\in \mathbb{N}^+$. Let $a_0$  be the smallest number of the set $\{a_m,\ m\in I'\}$. Now we consider the tensor product of the restriction sequence with $E_i^*\otimes E_j$:
$$
	\xymatrix{
		0 \ar[r] & E_i^{*}\otimes E_j(-1) \ar[r] & E_i^{*}\otimes E_j \ar[r] & E_i^{*}\otimes E_j|_H \ar[r] & 0.
	}
$$
The long exact sequence in cohomology gives 
$$
 \begin{tikzcd}
         \cdots \arrow[r]  & \Ext^1_{\CP^2}(E_i, E_j) \arrow[r] \arrow[d, phantom,""{coordinate,name=Z}]
            &   \Ext^1_H(E_i|_H,E_j|_H) \arrow[dll,
                                                            "",
                                                            rounded corners,
                                                            to path={ -- ([xshift=2ex]\tikztostart.east)
                                                        |- (Z) [near end]\tikztonodes
                                                        -| ([xshift=-2ex]\tikztotarget.west)
                                                        -- (\tikztotarget)}] \\
             \Ext^2_{\CP^2}\left(E_i, E_j(-1)\right) \arrow[r]
                & \Ext^2_{\CP^2}(E_i,E_j) \arrow[r] 
                    & \cdots
        \end{tikzcd}
    $$

Now, $\Ext^1_{\CP^2}(E_i,E_j) = \Ext^2_{\CP^2}(E_i,E_j) = 0 $ by the strong property of the exceptional collection due to the Theorem \ref{strong}, and $\Ext^2_{\CP^2}\left(E_i, E_j(-1)\right) = 0$ or $\BC$ by the previous argument, thus 
$$
\Ext^1_H(E_i|_H, E_j|_H) \cong \HO^1\left(\CP^1, E_i^{*}\otimes E_j|_H\right) = \HO^1(\CP^1, \bigoplus_{m\in I'} \CO(a_m)^{u_m}) = 0 \ \text{or } \BC. 
$$
Then we can deduce that $a_0 \geq -2$. Then since $n>0$ we have $ E_i^{*}\otimes E_j(n)|_H = \bigoplus \CO(a_m+n)^{u_m}$ where $a_m +n >-2$ for any $m\in I'$, thus
$$
	\HO^1\left(\CP^1,E_i^{*}\otimes E_j(n)|_H\right)  = \HO^1\left(\CP^1, \bigoplus_{m\in I'} \CO(a_i+n) \right)= 0.
$$
This finishes the proof of the lemma.
\end{proof}

\begin{proposition} \label{vanish3}
	Let $\BE =(E_0,E_1,E_2)$ be a full exceptional sequence consisting of sheaves on in $D^b(\CP^2)$. Suppose for a pair of indices $i,\ j \in \{0,1,2\}$ we have that
	$
		\mu(E_i)-\mu(E_j) \leq 2.
	$
	Then 
	$$
	\HO^k\left(\CP^2, E^*_i\otimes E_j (n)\right) = 0,
	$$
for $k>0$, $n\geq 0$. 
\end{proposition}
\begin{proof}
Let $H = \CP^1$ be a hyperplane in $\CP^2$. We argue by induction on $n$. For $n=0$, the claim follows from the strong property of the exceptional collection.  Now assume $\HO^k\left(\CP^2, E_i^{*}\otimes E_j(n-1)\right) = 0$ for $n>1$ and $k>0$, then we consider the short exact sequence
$$
	\xymatrix{
		0 \ar[r] & E_i^{*} \otimes E_j(n-1) \ar[r] & E_i^{*}\otimes E_j(n) \ar[r] & E_i^{*}\otimes E_j(n)|_H \ar[r] & 0.
	}
$$
By taking the associated long exact sequence and using the vanishing of $\HO^k\left(\CP^1,E_i^{*}\otimes E_j(n)|_H \right)$ for $n>0$ ($k=1$ follows from Lemma \ref{vanish_0} and $k>1$ always vanish because of dimension $1$), we have 
$$
	\HO^k\left(\CP^2, E_i^{*}\otimes E_j(n-1)\right) = \HO^k \left(\CP^2, E_i^{*}\otimes E_j(n)\right))=0.
$$
This proves the inductive step.

\end{proof}
We have the following useful exact sequence:
\begin{lemma}[Symmetric product of Euler sequence]
We write $\CP^2$ as $\CP(V)$, where $V$ is a linear space of dimension $3$ (or for general $n$). Then we have the short exact sequence of sheaves for $k\geq 1$:
\begin{equation}\label{sym-euler}
		\xymatrix{
			0\ar[r] & \sym^{k-1} \left(V\otimes \CO(1)\right) \ar[r] & \sym^k\left(V\otimes \CO(1)\right) \ar[r] & \sym^k\TT \ar[r] & 0.
		}
	\end{equation}
\end{lemma}
\begin{proof}
This follows directly from the following simple result in commutative algebra: suppose there is a short exact sequence of free $A$-modules
$$
\xymatrix{
0\ar[r] & A \ar[r]^{e}&  N \ar[r]^p & P \ar[r]& 0,
}
$$
then there is a short exact sequence
$$
    \xymatrix{
    0 \ar[r] & A\otimes_A  \sym^{k-1} N \ar[r]^-{\wtil{e}} & \sym^k N \ar[r]^-{\wtil{p}} & \sym^k P \ar[r] & 0
    }
$$
where $\wtil{e} := e\otimes id^{\otimes k-1}$ and $\wtil{p} := p^{\otimes k}$. Obviously $\wtil{e}$ is injective and $\wtil{p}$ is surjective. For an element $n_1\otimes\cdots \otimes n_k\in \sym^k N$ whose image under $\wtil{p}$ is $0$, we have at least one $n_i$ such that $p(n_i) = 0$. Since we are in the symmetric power of modules we can assume that $i=1$, so $n_1 = e(a)$ for some $a\in A$. Thus $n_1\otimes \cdots \otimes n_k$ is in the image of $\wtil{e}$. 
\end{proof}
\begin{corollary} \label{sym_euler_2}
    For a pair of vector bundles $E$ and $F$ on $\CP^2$, we suppose;. 
    $$
    \Ext^k_{\CP^2}\left(E,F(n)\right) = 0
    $$
    for $n\geq 0$, then 
    $$
    \Ext^k_X(\pi^*E, \pi^*F) = 0.
    $$
\end{corollary}
\begin{proof}
      We have 
      $$
      \Ext^k_X(\pi^*E,\pi^*F) = \Ext^k_{\CP^2}(E,F\otimes \sym^{\bullet}\TT) = \bigoplus_{n=0} \HO^k(\CP^2, E^*\otimes F \otimes \sym^n \TT)
      $$ 
      by adjunction and $\pi_* \CO_{X} \cong \sym^{\bullet} \TT$. For each $n\geq 1$ we tensor the exact sequence (\ref{sym-euler}) with $E^*\otimes F$, and note that  $\sym^n\left(V(1)\right) \cong \CO(n)^{\binom{n+2}{n}}$. Thus we get an exact sequence:
    $$
		\xymatrix{	
			0\ar[r] & \CO(n-1)^{\binom{n+1}{n-1}}\otimes E^*\otimes F \ar[r] & \CO(n)^{\binom{n+2}{n}}\otimes E^*\otimes F \ar[r] & \sym^n \TT\otimes E^*\otimes F\ar[r] & 0.
		}
	$$
	Then we take the long exact sequence of cohomology groups:
    $$
    \begin{tikzcd}
         \cdots \arrow[r]  & \Ext^1_{\CP^2}\left(E, F(n)\right)^{\oplus \binom{n+2}{n}} \arrow[r] \arrow[d, phantom,""{coordinate,name=Z}]
            &    \Ext^1_{\CP^2}\left(E,  F\otimes \sym^n \TT_{\CP^2} \right) \arrow[dll,
                                                            "",
                                                            rounded corners,
                                                            to path={ -- ([xshift=2ex]\tikztostart.east)
                                                        |- (Z) [near end]\tikztonodes
                                                        -| ([xshift=-2ex]\tikztotarget.west)
                                                        -- (\tikztotarget)}] \\
             \Ext^2_{\CP^2}\left(E, F(n-1)\right)^{\oplus \binom{n+1}{n-1}} \arrow[r]
                &  \Ext^2_{\CP^2}\left(E,F(n)\right)^{\oplus \binom{n+2}{n}} \arrow[r] 
                    & \cdots
        \end{tikzcd}
    $$
	so $\HO^k(\CP^2,E^*\otimes F \otimes \sym^n\TT) $ vanishes due to the assumptions.
\end{proof}
\begin{theorem}\label{tilt_obj}
	Let $\BE=(E_0,E_1,E_2)$ be full  exceptional collection consisting of sheaves in $D^b(\CP^2)$ such that $\mu(E_2)-\mu(E_0) \leq 2$,  then $T=\bigoplus^2_{i=0}\pi^*E_i$ is a tilting object in $D^b(X)$.
\end{theorem}
\begin{proof}
We firstly show that $\Ext^n_X(\pi^*T, \pi^* T) = 0$, for $n>0$. Since the Ext functor commutes with direct sum functor,  it is enough to show that for any $i,\ j$
$$
\Ext^n_X(\pi^* E_i, \pi^* E_j) = 0, \text{ for } n>0.
$$

By  Proposition \ref{vanish3} and Corollary \ref{sym_euler_2}, we have
	$$
		\HO^n\left(\CP^2, E^*_i\otimes E_j \otimes \sym^{\bullet}\TT\right) = 0,
	$$
	for $n>0$. Thus $\Ext^n_X(\pi^*T, \pi^* T) = 0$ for $n>0$.
	
	Combining with  Lemma \ref{gen} that $\pi^*T$ is a classical generator, we conclude that $\pi^* T$ is tilting object in $D^b(X)$.
\end{proof}
As a corollary, we have
\begin{corollary}[Main theorem I]\label{tilting}
	Let $\BE=(E_0,E_1,E_2)$ be a full exceptional collection consisting of sheaves in $D^b(\CP^2)$. We denote the corresponding helix by $\BH$. There exists a thread $\BE' = (E_i, E_{i+1},E_{i+2})$ in $\BH$ such that $\BE'$ satisfies the conditions in Proposition \ref{vanish3}. Thus the pull back of the direct sum of the objects in $\BE'$ is a tilting object in $D^b(X)$.
\end{corollary}
\begin{proof}
	 We denote  the corresponding helix  by $\BH = (\cdots,E_0, E_1, E_2, E_3, E_4, \cdots)$. Since the Serre functor on $\CP^2$ is $-\otimes \CO(-3)[2]$, we have 
	 $$
	 \mu(E_{i-3}) = \mu\left(E_i \otimes \CO(-3)\right)  = \mu(E_i) - 3.
	 $$
	 
	 We can consider the numbers $a_0 = \mu(E_2)-\mu(E_0)$, $a_1=\mu(E_3)-\mu(E_1)$ and $a_2=\mu(E_4)-\mu(E_2)$. By Lemma \ref{inc}, we  actually have that the slopes in the helix are strictly increasing, thus  $a_i >0$. The sum $a_0+a_1+a_2 = \mu(E_4)-\mu(E_1) + \mu(E_3)-\mu(E_0) = 6$. Then  one of the $a_i$ must be  $\leq 2$. Then we let $T = \bigoplus_{j=i}^{i+2}\pi^* E_j$, and due to Theorem \ref{tilt_obj}, we have $T$ a tilting object in $D^b(X)$.
\end{proof}
\section{Secondary Quiver and Universal Extension}
Let $\BE=(E_0,E_1,E_2)$ be a full  exceptional collection consisting of sheaves in $D^b(\CP^2)$ (then $\BE$ is strong by Theorem \ref{strong}), and we denote $T_i := \pi^* E_i$, $T = \oplus_i T_i$. In general, $T$ will not be a tilting object, since we  may have  $\Ext^n_X(T,T) \neq 0$, for $n>0$. 
\begin{example}
Let $\BE = \left(\Omega(1), \CO,\CO(2)\right)$, note that $\mu\left(\CO(2)\right) - \mu\left(\Omega(1)\right) = 5/2$. Then
\begin{eqnarray*}
    \Ext^1_X\left(\pi^*\CO(2),\pi^*\Omega(1)\right) &\cong&  \BC^3.
\end{eqnarray*}
so $T= \pi^* \Omega(1)\oplus \pi^* \CO \oplus \CO(2)$ is not a tilting object.
\end{example}
However, in general we have
\begin{lemma}\label{van_geq2}
	For any $i,\ j = 0,\ 1,\ 2$, we have
	\begin{enumerate}
	\item 
	    $$
		    \Ext^{\geq 2}_X(T_i, T_j) = 0,
    	$$
    \item 
        if $i\leq j$ then
        $$
            \Ext^{\geq 1}_X(T_i,T_j) = 0.
        $$
    \end{enumerate}
\end{lemma}
\begin{proof}
	By Corollary \ref{sym_euler_2}, it is enough to show that for any $i,\ j$, 
	$$
		\HO^n\left(\CP^2, E_i^*\otimes E_j (k)\right) = 0,\ n \geq 2,
	$$
	where $k\geq 0$. By restricting the bundle $E_i^*\otimes E_j$ to a degree $k$ curve  $C$ in $\CP^2$ where $k>0$, we have the short exact sequence:
	$$
		\xymatrix{
		0\ar[r] & E_i^*\otimes E_j \ar[r] & E_i^*\otimes E_j(k) \ar[r] & E_i^*\otimes E_j(k) |_C  \ar[r] & 0
		}
	$$
	Then  we take the cohomology groups, and use the vanishing of $\HO^{\geq 2}(C, E_i^*\otimes E_j(k)|_C) = 0$ so we proved the first statement.
	
	The second statement follows from a combination of Lemma \ref{inc} and Proposition \ref{vanish3}.
\end{proof}

Now we package the non-tilting information of $T$ into the following algebraic structure, and we call it the secondary quiver:
\begin{definition}[Secondary quiver] 
Let $\BE$ be a full exceptional collection on $\CP^2$ consisting of sheaves and $T_i := \pi^* E_i$. Then let $Q_{\BE}$ be the quiver with $3$ vertices indexed by $0,\ 1,\ 2$, and the  arrows from $i$ to $j$ correspond to a basis of $\Ext^1_X(T_j,T_i)$.  The  quiver $Q_{\BE}$ is called the secondary quiver of $\BE$. The category of right finitely generated module of $\BC Q_{\BE}$ is denoted by\\  $\Rep\ Q_{\BE} = \mod \BC Q_{\BE}$. We denote the simple modules in $\Rep\ Q_{\BE}$ by $S_i$.
\end{definition}
\begin{remark}
With the notations above, we have $\Ext^1_{\BC Q_{\BE}}(S_i , S_j) \cong \Ext^1_X(T_i,T_j)$.
\end{remark}
\begin{proposition} \label{class_of_2nd}
	 $Q_{\BE}$ could only be one of the following forms:
	\begin{enumerate}
		\item
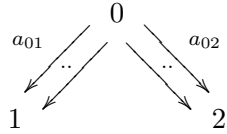
			$$
				\xymatrix{
					& 0 \ar@<1ex>[dl] \ar@{}[dl]|-{..} \ar@<-1ex>[dl]_{a_{01}} \ar@<1ex>[dr]^{a_{02}} \ar@{}[dr]|-{..} \ar@<-1ex>[dr] & \\
					1  & & 2 
									}
			$$
			where $a_{0i}\geq 0,\ i\in \{1,\ 2\}$ denote the number of arrows. Then any object $M$ in $\Rep\  Q_{\BE}$ fits into the short exact sequence:
			$$
				\xymatrix{
					0\ar[r] & S_0^{m_0} \ar[r] & M \ar[r] & S_1^{m_1} \oplus S_2^{m_2} \ar[r] & 0
				}
			$$
		    where $m_i \in \mathbb{N}$
		\item 
			$$
				\xymatrix{
					0  \ar@<1ex>[dr] \ar@{}[dr]|-{..} \ar@<-1ex>[dr]_{a_{02}}&   & 1 \ar@<1ex>[dl]^{a_{12}} \ar@{}[dl]|-{..} \ar@<-1ex>[dl]\\
				 & 2 & 
				}
			$$
			where $a_{i2}\geq 0, \ i\in \{0,\ 1\}$ denote the number of arrows. Then any object $M$ in $\Rep\  Q_{\BE}$ fits into the short exact sequence:
	       $$
				\xymatrix{
					0 \ar[r] & S_0^{m_0} \oplus S_1^{m_1} \ar[r] & M \ar[r] & S_2^{m_2} \ar[r] & 0
				}
			$$
			$m_i \in \mathbb{N}$.
	\end{enumerate}
\end{proposition}
From now on, we call the first quiver case 1, and the second quiver case 2.

\begin{proof}
	There is no arrow from $i$ to $j$ for $i\geq j$ due to Lemma \ref{van_geq2}. The only case we need to exclude is 
	$$
		\xymatrix{
			0 \ar@<1ex>[r]^{a_{01}} \ar@{}[r]|-{:} \ar@<-1ex>[r] & 1 \ar@<1ex>[r]^{a_{12}} \ar@{}[r]|-{:} \ar@<-1ex>[r] & 2
		}
	$$
	where $a_{21}, a_{10} >0$. By definition of the arrows in the secondary quiver, this is equivalently to say  
	$$
		\Ext^1_X(T_i,T_{i-1}) \neq 0,\ i=2,\ 1.
	$$
	Using the contraposition of the Proposition \ref{vanish3} we have $\mu(E_i)-\mu(E_{i-1}) > 2$, thus 
	\begin{equation} \label{slope1}
		\mu(E_2) -\mu(E_0) = \left(\mu(E_2)-\mu(E_1)\right) + \left(\mu(E_1)-\mu(E_0)\right)>4.
	\end{equation}
    However, consider the helix $\BH = (\cdots, E_0, E_1, E_2, E_3,\cdots)$ generated by $\BE$. We have $\mu(E_3) = \mu\left(E_0(3)\right) = \mu(E_0)+3$, and $\mu(E_2) < \mu(E_3) = \mu(E_0)+3$ by Lemma \ref{inc}. So (\ref{slope1}) contradicts with the above facts.
\end{proof}

Following an idea in \cite[Section 3]{HP14}, we construct the tilting objects from the secondary quiver $Q_{\BE}$. The procedure is called universal extension in loc. cit. and we will use that name.
\begin{definition}[Universal extension] \label{uni_ext}
	 Let $\BE=(E_0,E_1,E_2)$ be a full  exceptional collection consisting of sheaves in $D^b(\CP^2)$ (then $\BE$ is strong by Theorem \ref{strong}),  $T_i := \pi^* E_i$ and $T = \oplus_i T_i$.
	 \begin{enumerate}
	 	 \item 
	    For the secondary quiver of case 1 in Proposition \ref{class_of_2nd}, we consider the objects $\wtil{T_m}\in \Coh X$ ($m=1,\ 2$) which are obtained from the extensions:
	     \begin{equation}\label{univ_exten_1}
	    	\xymatrix{
		    	0 \ar[r] & T_0\otimes\Ext^1_X(T_m, T_0)^*  \ar[r] & \wtil{T_m}     \ar[r] &  T_m \ar[r]& 0.
	    	}
	     \end{equation}
     where the connection map is the dual of the canonical evaluation map in $D^b(X)$:
    $$
       can: T_m\otimes \Ext^1_X(T_m,T_0) \longrightarrow T_0[1].
    $$
	 \item
	 For the secondary quiver of case 2 in Proposition \ref{class_of_2nd}, we consider the object $\wtil{T_2}\in \Coh X$, which is obtained from the following extension:
	 \begin{equation}
	 	\xymatrix{
			0 \ar[r] & \bigoplus_{i=0, 1} T_i\otimes \Ext^1_X(T_2,T_i)^* \ar[r] & \wtil{T_2} \ar[r] & T_2 \ar[r] & 0
		}
	 \end{equation}
	 where the connection map is the dual of the canonical evaluation map in $D^b(X)$:
	 $$
	    can: \oplus_{i=0,1} T_2\otimes \Ext^1_X(T_2,T_i) \longrightarrow \oplus_{i=0,1}T_i[1].
	 $$
    \end{enumerate}
\end{definition}

\begin{theorem}\label{obj_tilt2}
    For the secondary quiver of case 1 in Proposition \ref{class_of_2nd}, the object $\wtil{T}:=T_0\oplus \wtil{T_1}\oplus \wtil{T_2}$ is a tilting object in $D^b(X)$.
    
    For the secondary quiver of case 2 in Proposition \ref{class_of_2nd}, the object $\wtil{T} :=T_0\oplus T_1\oplus \wtil{T_2}$ is a tilting object in $D^b(X)$.
\end{theorem}
\begin{proof}
    We prove the case 1, the proof of the case 2 is similar.
     \begin{enumerate}
      \item
    $\wtil{T}$ generates the whole category, since $T_i \in \langle \wtil{T} \rangle$ by definition, and $\langle T=\oplus_i T_i \rangle = D^b(X)$ by Lemma \ref{gen}.
    
    \item   We firstly show that $\Ext^{\geq 1}_X (T_0,\wtil{T_m}) = 0$ for $m=1,\ 2$. We take the long exact sequence of (\ref{univ_exten_1}) associated with the functor $\Hom(T_0,-)$:
    \begin{equation*} \label{47long1}
        \begin{tikzcd}
           \cdots \arrow[r] 
            &  \Ext^1_X(T_0,\wtil{T_m}) \arrow[r] \arrow[d,phantom,""{coordinate, name=Y}] 
                & \Ext^1_X(T_0,T_m)   \arrow[dll,
                                                            "",
                                                            rounded corners,
                                                            to path={ -- ([xshift=2ex]\tikztostart.east)
                                                        |- (Y) [near end]\tikztonodes
                                                        -| ([xshift=-2ex]\tikztotarget.west)
                                                        -- (\tikztotarget)}] \\
           \Ext^2_X(T_0,T_0)\otimes\Ext^1_X(T_m,T_0)^* \arrow[r] 
            & \Ext^2_X(T_0, \wtil{T_m}) \arrow[r]
                & \cdots  .
        \end{tikzcd}
    \end{equation*}
    We have $\Ext^k_X(T_0,T_0) =0$ and $\Ext^k_X(T_0,T_m) = 0$ ($k\geq 1$) because of the second statement of Lemma \ref{van_geq2}. Thus $\Ext^{\geq 1}_X (T_0,\wtil{T_m}) = 0$ since the terms on the two sides both vanish. 
    
    \item Now we show that $\Ext^{\geq 1}_X( \wtil{T_m},T_0) = 0$ for $m=1,\ 2$. By taking the long exact sequence associated with the functor $\Hom_X(-,T_0)$:
    $$
        \begin{tikzcd}
           & \cdots \arrow[r] \arrow[d, phantom,""{coordinate,name=Z}]
            &   \Hom_X(T_0,T_0)\otimes \Ext^1_X(T_m,T_0)  \arrow[dll,
                                                            "can_*",
                                                            rounded corners,
                                                            to path={ -- ([xshift=2ex]\tikztostart.east)
                                                        |- (Z) [near end]\tikztonodes
                                                        -| ([xshift=-2ex]\tikztotarget.west)
                                                        -- (\tikztotarget)}] \\
             \Ext^1_X(T_m,T_0) \arrow[r]
                & \Ext^1_X(\wtil{T_m},T_0) \arrow[r] \arrow[d,phantom,""{coordinate, name=Y}]
                    &\Ext^1_X(T_0,T_0)\otimes \Ext^1_X(T_m,T_0) \arrow[dll,
                                                            "",
                                                            rounded corners,
                                                            to path={ -- ([xshift=2ex]\tikztostart.east)
                                                        |- (Y) [near end]\tikztonodes
                                                        -| ([xshift=-6ex]\tikztotarget.west)
                                                        -- (\tikztotarget)}] \\
            \Ext^2_X(T_m,T_0) \arrow[r]
                & \Ext^2_X(\wtil{T_m},T_0) = 0 \arrow[r] 
                    & \cdots
        \end{tikzcd}
    $$
    Note that the connection map is the Yonneda product\cite[Definition 4.16]{May}
    \begin{eqnarray*}
        can_*: \Hom_X(T_0,T_0)\otimes \Ext^1_X(T_m,T_0) &  \longrightarrow& \Ext^1_X(T_m,T_0) \\
        id \otimes \eta &\longmapsto & \eta.
    \end{eqnarray*}
    Thus $can_*$ is surjective, then we can see that $\Ext^1_X(\wtil{T_m},T_0) = 0$. 
    
    Since $\Ext^{\geq 2}_X(T_m,T_0) = 0$ due to the first statement of Lemma \ref{van_geq2} and $\Ext^{\geq 1}_X(T_0,T_0) = 0$, we have that $\Ext^{\geq 2}_X(\wtil{T_m},T_0) = 0$ since terms on the two sides both vanish.

    \item Then we check that $\Ext^{\geq 1}_X(\wtil{T_m}, \wtil{T_m}) = 0$. By taking the long exact sequence associated with functor $\Hom_X(\wtil{T_m},-)$ ($i\geq 1$):
        \begin{equation}\label{47long2}
             \rightarrow  \Ext^i_X(\wtil{T_m},T_0)\otimes \Ext^1_X(T_m,T_0)^* \rightarrow \Ext^i_X(\wtil{T_m},\wtil{T_m}) \rightarrow \Ext^i_X(\wtil{T_m}, T_m) \rightarrow
        \end{equation}
        since $\Ext^i_X(\wtil{T_m},T_0 ) = 0$ by the above arguments,  so it is equivalent to check the vanishing of $ \Ext^i_X(\wtil{T_m},T_m)$. 
        
        Again we take the  functor $\Hom_X(-,T_m)$ of the universal extension sequence:
        $$
                \rightarrow \Ext^i_X(T_m,T_m)  \rightarrow \Ext^i_X(\wtil{T_m},T_m)    \rightarrow \Ext^i_X(T_0,T_m)\otimes \Ext^1_X(T_m,T_0) \rightarrow
        $$
        note that $\Ext^{\geq 1}_X(T_0,T_m) = 0$ since $0< m=1,\ 2$ and $\Ext^i_X(T_m,T_m) = 0$ because of Lemma \ref{van_geq2}. 
       The cohomology groups on the both sides vanish, so we have $\Ext^i(\wtil{T_m},T_m) = 0$ for $i\geq 1$. Return back to (\ref{47long2}), and we obtain that $\Ext^i_X(\wtil{T_m},\wtil{T_m}) = 0$.
    \item Finally we check that $\Ext^{\geq 1} _X(\wtil{T_m},\wtil{T_n}) = 0$ when $m,\ n \in\{1,\ 2\}$ ($m\neq n$).  Consider the universal extension of $T_n$ and we take the  functor $\Hom_X(-,\wtil{T_m})$, then we have a long exact sequence:
        $$
\rightarrow \Ext^i_X(T_n, \wtil{T_m}) \rightarrow \Ext^i_X(\wtil{T_n},\wtil{T_m}) \rightarrow
                        \Ext^i_X(T_0,\wtil{T_m})\otimes \Ext^1_X(T_n,T_0) 
                        \rightarrow
        $$
        because $\Ext^{\geq 1}_X(T_0,\wtil{T_m} ) = 0 $ by the previous arguments, so it is equivalent to check that $\Ext^i_X(T_n,\wtil{T_m}) = 0$.
        
        Again we take the functor $\Hom_X(T_n,-)$ to the universal extension of $T_m$:
        $$  
            \rightarrow \Ext^i_X(T_n,T_0)\otimes \Ext^1_X(T_m,T_0)^* \rightarrow \Ext^i_X(T_n,\wtil{T_m}) \rightarrow \Ext^i_X(T_n,T_m)  \rightarrow
        $$
        since $\Ext^{\geq 1}_X (T_n,T_m) = 0$ by our assumption of  case $1$ , thus we obtain that $$\Ext^i_X(T_n, \wtil{T_m}) = 0.$$
\end{enumerate}
    
\end{proof}

\section{Functor $\CF_Q$}
The object of this section is to construct a functor from the category of representations of the secondary quiver $\Rep Q_{\BE}$ to $\Coh(X)$, and we obtain that the functor sends tilting objects to tilting objects and simple modules $S_i$ to $T_i$.

We have the following simple lemma:
\begin{lemma} \label{alg_mor}
  Let $\BE$ be a full exceptional collection on $\CP^2$ consisting of sheaves, then there is an injective map of algebras  from $\BC Q_{\BE}$ to $\End \wtil{T}$ where $\wtil{T}$ is the object obtained by the universal extension as in the Definition \ref{uni_ext} and Theorem \ref{obj_tilt2}:
\end{lemma}
\begin{proof}
   The algebra $\End \wtil{T}$ contains three idempotents $\wtil{e_i}$ which are projection maps $\wtil{T} \rightarrow \wtil{T_i}$. We denote the idempotents in $\BC Q_{\BE}$ by $e_i$. 
   \begin{enumerate}
       \item For case 1,

            $$
				
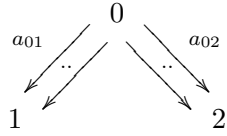
\xymatrix{
					& 0\ar@<1ex>[dl] \ar@{}[dl]|-{..} \ar@<-1ex>[dl]_{a_{01}} \ar@<1ex>[dr]^{a_{02}} \ar@{}[dr]|-{..} \ar@<-1ex>[dr] & \\
					1 & & 2 
									}
			$$
	
	 define a map $\iota: \BC Q_{\BE} \longrightarrow \End \wtil{T}$ which sends $e_i $ to  $  \wtil{e_i}$; and the set of arrows between $0$ and $m$ ($m= 1,\ 2$) to one  component of
	$$
	   \begin{tikzcd}
	         0 \arrow[r] 
	            & T_0\otimes\Ext^1_X(T_m,T_0)^*\arrow[r]
	                & \wtil{T_m}
        \end{tikzcd}
	$$
	 in the universal extension sequence.
	 
	Then $\iota$ is a map of algebras. $\iota$ is injective since $\Ext^1_X(T_m,T_0)\subset \Hom_X(T_0, \wtil{T_m})$ by definition of universal extension sequence;
	\item similarly for case 2,
			$$
				
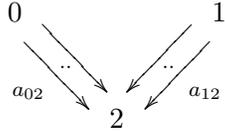
\xymatrix{
					0  \ar@<1ex>[dr] \ar@{}[dr]|-{..} \ar@<-1ex>[dr]_{a_{02}}&   & 1 \ar@<1ex>[dl]^{a_{12}} \ar@{}[dl]|-{..} \ar@<-1ex>[dl]\\
				 & 2 & 
				}
			$$
		then we define a map  $\iota: \BC Q_{\BE} \longrightarrow \End \wtil{T}$ which sends $e_i$ to $\wtil{e_i}$ and the set of arrows between $i$ and $2$ ($i= 0,\ 1$) are mapped to one component of
		$$
		\begin{tikzcd}
	         0 \arrow[r] 
	            & \bigoplus_{i=0,1}T_i\otimes \Ext^1_X(T_2,T_i)^* \arrow[r]
	                & \wtil{T_2}
        \end{tikzcd}
		$$
	    in the universal extension sequence.
	  
	\end{enumerate}
\end{proof}

\begin{definition}
We denote $A:= \BC Q_{\BE}$ and $B:=\End \wtil{T}$, where $\wtil{T}$ is defined in theorem \ref{obj_tilt2}. Then we define a functor $\CF_Q$ from 
\begin{eqnarray*}
	\CF_Q: \text{mod-} A &\longrightarrow& \Coh \left(X\right) 
\end{eqnarray*}
by the composition:
$$
    \xymatrix{
        \text{mod-}A \ar[r]^{-\otimes B} & \text{mod-}B \ar[r]^-{\otimes_B \wtil{T} }  & \Coh(X)
    }
$$
\end{definition}
We have
\begin{proposition}
    $\CF_Q$ is an exact functor.
\end{proposition}
\begin{proof}
     We show the result only for case 1 in Proposition \ref{class_of_2nd}, since  case 2 is quite similar. In this case $\wtil{T} = T_0\oplus \wtil{T_1}\oplus \wtil{T_2}$.
    
    Since $\text{mod-}A$ is a category of finite length (where each object has a finite Jordan-H\"{o}lder filtration), we only need to show that the derived functor of $\CF_Q$ sends  $S_i$ to the objects in $\Coh X$, that is, it preserves the t-structure. 
    
    We consider the left derived functor of $\CF_Q$:
    $$
        \begin{tikzcd}
            \mathbb{L} \CF_Q : D^b(\text{mod-}A) \arrow[r,"-\otimes^{\BL}_A B"] 
                & D^b(\text{mod-} B) \arrow[r,"-\otimes^{\BL}_B \wtil{T}"] 
                    & D^b(X) 
        \end{tikzcd}
    $$
    
    In this case $S_0 = P_0$ where $P_0:= e_0 A$. We also denote the canonical projective modules in $\text{mod-}B$ by $\wtil{P_i} = \wtil{e_i} B \cong \Hom_X(\wtil{T}, \wtil{T_i})$ for $i = 1,\ 2$ and $\wtil{P_0} \cong \Hom_X(\wtil{T}, T_0)$. Then by definition
    $$
        \BL \CF_Q (P_0) \cong \wtil{P_0} \otimes^{\BL}_B \wtil{T} \cong \wtil{T_0}
    $$
    in $D^b(X)$. Thus $\CF_Q$ is exact on $S_0$.
    
    For $S_m$ where $m=1,\ 2$. We have distinguished triangle in $D^b(\text{mod-}A)$:
    $$
        \xymatrix@C=4em{
            S_0^{a_{0m}} = S_0\otimes \Ext^1_X(S_m,S_0)^* \ar[r] & P_m \ar[r] & S_m \ar[r]^-{\eta=\sum_k x^k_{0m}.} & S_m[1]
        }
    $$
    where $x^k_{0m}$ denotes the arrow from the vertex $0$ to $m$.
    
    By taking the functor $-\otimes^{\BL}_A B$,  we have a distinguished triangle in $D^b(\text{mod-}B)$:
    \begin{equation}\label{exact1}
        \xymatrix@C=4em{
          \wtil{P_0}^{a_{0m}} \ar[r] & \wtil{P_m} \ar[r] & S_m\otimes^{\BL}_A B \ar[r]^{\iota\left(\sum_k x^k_{0m} \right).} & \wtil{P_0}^{a_{0m}}[1]
        }
    \end{equation}

    Then by taking the functor $-\otimes_B^{\BL} \wtil{T}$, and under the canonical isomorphisms $\wtil{P_i}\otimes_B \wtil{T} \cong \wtil{T_i}$, we have morphisms between distinguished triangles in $D^b(X)$:
    $$
        \xymatrix@C=4em{
        \wtil{P_0}^{a_{0m}}\otimes \wtil{T} \ar[d]_{\cong} \ar[r] & \wtil{P_m}\otimes \wtil{T} \ar[d]_{\cong}\ar[r] & S_m\otimes_A B\otimes_B \wtil{T} \ar[d]^f \ar[r]^-{\CF_Q(\sum_k x^k_{0m}.)} & \wtil{P_0}^{a_{0m}}\otimes \wtil{T}[1]\ar[d] \\
        \wtil{T_0}^{a_{0m}} \ar[r] & \wtil{T_m} \ar[r] & T_m \ar[r]^-{can^*} &  \wtil{T_0}^{a_{0m}}[1]
        }
    $$
 the diagram is commutative since $\CF_Q(\sum_k x^k_{0m}.) \cong can^*$ by definition. 
 $f$ is induced from the morphisms between the distinguished triangles. Thus we get $T_m \cong S_m\otimes_A B\otimes_B \wtil{T} =  \CF_Q(S_m)$ in $D^b(X)$. Moreover, since the lower triangle all lies in the canonical heart of $D^b(X)$:  $\Coh(X)$. Thus we actually have a short exact sequence:
    $$
        \xymatrix{
            0\ar[r] & \CF_Q(S_m)^{a_{0m}}\cong \wtil{T_0}^{a_{0m}} \ar[r] & \CF_Q(P_m) \cong  \wtil{T_m} \ar[r] & \CF_Q(S_m)\cong T_m \ar[r] & 0.
        }
    $$
    So we have tested the exactness on any simple module $S_m$ in $\text{mod-} A$, we have shown the exactness of $\CF_Q$.
\end{proof}
The proof in the above theorem also implies
\begin{corollary}\label{corr_simp}
    For simple modules $S_i\in \text{mod-} A$, we have $\CF_Q(S_i) = T_i$.
\end{corollary}

Here we say $M$ is a tilting object in some abelian category, we mean $M$ is a tilting object in its corresponding derived category.
\begin{theorem}\label{F_Q}
	Let $\BE$ be a full exceptional collection consisting of sheaves on $\CP^2$, $Q_{\BE}$ be its secondary quiver and $A$ be the path algebra of $Q_{\BE}$. Then $\CF_Q$ sends tilting objects in  $\text{mod-} A$ to tilting objects in $\Coh (X)$.
\end{theorem}
\begin{proof}
	We prove  case 1, the proof of case 2 is essentially the same. Firstly we show the property of generating. Let $R\in \text{mod-} A$ be a generator for $D^b(\text{mod-}A)$, then $S_i\in \langle R \rangle$ for any $i$. Then under the correspondence of $\CF_Q$, we have 
	$$
	    \CF_Q(S_i) = T_i\in \langle \CF_Q(R) \rangle
	$$
	in $D^b(X)$ due to Corollary \ref{corr_simp}. Now because $\bigoplus_{i=0}^2 T_i$ generates $D^b(X)$, thus we have  that $\CF_Q(R)$ generates the whole category.

	Given a tilting object $R\in \text{mod-}A$, we assume $R$ fits into 
	\begin{equation}\label{ses_R}
		\xymatrix{
			0 \ar[r] & S_0^{m_0} \ar[r] & R \ar[r]& S_1^{m_1}\oplus S_2^{m_2} \ar[r] & 0,
		}
	\end{equation}
	and $\Ext^1_A(R,R) = 0$. For any pair of objects $M,\ N\in\text{mod-}A$, $\Ext^{\geq 2}_A(M,N) = 0$ since $A$ is the path algebra of an acyclic quiver without relations.  For the corresponding object $\wtil{T} = \CF_Q(R)$ under the above definition,  we have 
	\begin{equation} \label{ses_T}
		\xymatrix{
			0 \ar[r] & T_0^{m_0} \ar[r] & \wtil{T}\ar[r] & T_1^{m_1}\oplus T_2^{m_2} \ar[r] & 0.
		}
	\end{equation}
    By applying the functors $\Hom_X(-, T_1^{m_1}\oplus T_2^{m_2})$, $\Hom_X(-,T_0^{m_0})$ and $\Hom_X(-, \wtil{T})$ to (\ref{ses_T}), we get the corresponding long exact sequences and put them in the first, second and third row separately. Finally we get the commutative diagram: 
	$$
	\resizebox{\displaywidth}{!}{%
		\xymatrix@R=3em{
		 0 \ar[r] \ar[d] 
		    & \End_X(T_1^{m_1}\oplus T_2^{m_2}) \ar[r]\ar[d]^-{\alpha_1}
		        & \Hom_X(\wtil{T},T_1^{m_1}\oplus T_2^{m_2}) \ar[r] \ar[d]^{\gamma_1} 
		            & \Hom_X(T_0^{m_0}, T_1^{m_1}\oplus T_2^{m_2}) \ar[d] \\
		 \End_X(T_0^{m_0}) \ar[r]^-{\beta_1} \ar[d] 
		    & \Ext^1_X(T_1^{m_1}\oplus T_2^{m_2}, T_0^{m_0}) \ar[d]^{\alpha_2}\ar[r]^{\beta_2} 
		        & \Ext^1_X(\wtil{T},T_0^{m_0}) \ar[r]\ar[d]^{\gamma_2} 
		            & \Ext^1_X(T_0^{m_0},T_0^{m_0}) = 0 \ar[d] \\
		 \Hom_X(T_0^{m_0},\wtil{T}) \ar[r]^-{\theta_1} 
		    & \Ext^1_X(T_1^{m_1}\oplus T_2^{m_2}, \wtil{T}) \ar[r]^-{\theta_2} 
		        & \Ext^1_X(\wtil{T},\wtil{T}) \ar[r]\ar[d] 
		            & \Ext^1_X(T_0^{m_0},\wtil{T}) = 0  \\
		    &	
		        &	\Ext^1_X(\wtil{T}, T_1^{m_1}\oplus T_2^{m_2}) = 0 
		             &
		}
		}%
	$$
    where the morphisms on the columns are obtained in the following way: we first apply the functors $\Hom_X(T_0^{m_0},-)$, $\Hom_X(T_1^{m_1}\oplus T_2^{m_2},-)$, $\Hom_X(\wtil{T},-)$ and $\Hom_X(T_0^{m_0},-)$ on (\ref{ses_T}) separately, then take the  long exact  sequences. The diagram is commutative due to the naturality of derived functors. 
    
    We have a chain complex (not necessarily exact)
	\begin{equation} \label{ses1}
		\xymatrix@C=3em{
			\End_X(T_1^{m_1}\oplus T_2^{m_2}) \bigoplus \End_X(T_0^{m_0}) \ar[r]^-{         \eta = \binom{\alpha_1}{\beta_1}} 
			    & \Ext^1_X(T_1^{m_1}\oplus T_2^{m_2}, T_0^{m_0}) \ar[r]^-{f = \theta_2 \circ \alpha_2}_-{=\gamma_2\circ \beta_2} 
			        & \Ext^1_X(\wtil{T},\wtil{T})
		}
	\end{equation}
	 where $f$ is surjective. To obtain $\Ext^1_X(\wtil{T},\wtil{T}) = 0$, we need to show that $\eta$ is surjective, since $f\circ\eta = 0$. So we consider the corresponding diagram in $\Rep Q_{\BE}$ by applying Hom functors similarly to (\ref{ses_R}):
	 $$\resizebox{\displaywidth}{!}{%
		\xymatrix{
		 0 \ar[r] \ar[d] 
		    & \End_A(S_1^{m_1}\oplus S_2^{m_2}) \ar[r]^{\wtil{\phi}}\ar[d]^{\wtil{\alpha_1}} 
		        & \Hom_A(R,S_1^{m_1}\oplus S_2^{m_2}) \ar[r] \ar[d]^{\wtil{\gamma_1}}
		            & \Hom_A(S_0^{m_0}, S_1^{m_1}\oplus S_2^{m_2})=0 \ar[d]\\
		 \End_A(S_0^{m_0}) \ar[r]^{\wtil{\beta_1}} \ar[d] 
		    & \Ext^1_A(S_1^{m_1}\oplus S_2^{m_2}, S_0^{m_0}) \ar[d]^{\wtil{\alpha_2}}\ar[r]^{\wtil{\beta_2}} 
		        & \Ext^1_A(R,S_0^{m_0}) \ar[r]\ar[d]^{\wtil{\gamma_2}} 
		            & \Ext^1_A(S_0^{m_0},S_0^{m_0}) = 0 \ar[d]\\
		 \Hom_A(S_0^{m_0},R) \ar[r]^{\wtil{\theta_1}} \ar[d]
		    & \Ext^1_A(S_1^{m_1}\oplus S_2^{m_2}, R) \ar[r]^{\wtil{\theta_2}} \ar[d] 
		        & \Ext^1_A(R,R) \ar[r]\ar[d] 
		            & \Ext^1_A(S_0^{m_0},R) = 0  \\
		\Hom_A(S_0^{m_0},S_1^{m_1}\oplus S_2^{m_2}) = 0 
		    &	\Ext^1_A(S_1^{m_1}\oplus S_2^{m_2}, S_1^{m_1}\oplus S_2^{m_2})=0
		        &	\Ext^1_A(R, S_1^{m_1}\oplus S_2^{m_2}) = 0 
		            &
		}
		}%
	$$
	
	As before, we have a chain complex:
	 $$
	 	\xymatrix@C=3em{
			\End_A(S_1^{m_1}\oplus S_2^{m_2}) \bigoplus \End_A(S_0^{m_0}) \ar[r]^-{\wtil{\eta} = \binom{\wtil{\alpha_1}}{\wtil{\beta_1}}}  
			    & \Ext^1_A(S_1^{m_1}\oplus S_2^{m_2},S_0^{m_0})\ar[r]^-{\wtil{f} = \wtil{\theta}\circ\wtil{\alpha_2}} 
			        & \Ext^1_A(R,R)=0,
				}
	 $$
	but now it is exact at the middle term by diagram chasing: for any $q \in ker(\wtil{f})$, there exists $s\in \Hom_A(S_0^{m_0},R)$ such that $\wtil{\theta_1}(s) = \wtil{\alpha_2}(q)$. Since $\End_A(S_0^{m_0}) \cong \Hom_A(S_0^{m_0},R)$, we denote $p_1$ the corresponding element of $s$ in $\End_A(S_0^{m_0})$. Now consider $q' = q- \wtil{\beta_1}(p_1)$, then we have $\wtil{\alpha_2} (q') = 0$. By the exactness we have $p_2\in \End_A(S_1^{m_1}\oplus S_2^{m_2})$ such that $\wtil{\alpha_1}(p_2)=q'$. So we proved that exists an element $(p_1, p_2) \in	\End_A(S_1^{m_1}\oplus S_2^{m_2}) \bigoplus \End_A(S_0^{m_0})$ such that $\wtil{\eta}(p_1,p_2) = q$.

	Thus $\wtil{\eta}$ is surjective since $\Ext^1_A(R,R) = 0$ by assumption. Comparing the above exact sequence with the chain complex (\ref{ses1}), we have a commutative diagram:
	$$
		\xymatrix{
			\End_A(S_1^{m_1}\oplus S_2^{m_2}) \bigoplus \End_A(S_0^{m_0}) \ar[r]^-{\wtil{\eta}} \ar[d]^{\CF_Q}  
			    & \Ext^1_A(S_1^{m_1}\oplus S_2^{m_2},S_0^{m_0})\ar[d]^{\cong} \\
			\End_X(T_1^{m_1}\oplus T_2^{m_2}) \bigoplus \End_X(T_0^{m_0}) \ar[r]^-{\eta} 
			    & \Ext^1_X(T_1^{m_1}\oplus T_2^{m_2},T_0^{m_0})   
			}
	$$	 
	 Now $\eta$ is surjective since $\wtil{\eta}$ is surjective. Thus by the above arguments,  we have shown that $\Ext^1_X(\wtil{T},\wtil{T}) = 0$.
	 
	 The higher cohomology $\Ext^{>1}_X(\wtil{T},\wtil{T})$ vanishes since $\Ext^{>1}_X(T,T) = 0$.
\end{proof}

\section{Simple tilts and Ext-quivers}
\subsection{t-structures in $D^b(X)$}\label{cal_tilt}
 In the final section we calculated some simple tilts from a finite length heart in $D^b_0(X)$, where we denote by  $D^b_0(X)$ the full subcategory in $D^b(X)$ consisting of objects whose cohomology sheaves are set-theoritically supported on the zero section $\CP^2\subset \Tot \Omega_{\CP^2}$. Similarly, for some noetherian graded algebra $W$ we denote by $D^b_0(\text{mod-}W)$ the full subcategory consisting of objects whose cohomology modules are nilpotent. The derived equivalence (\ref{d_morita}) induced by the tilting object will restrict to the following case:
\begin{lemma}
Let $T$ be a tilting object in $D^b(X)$ and $B= \End(T)$, then the equivalence 
$$
    \RHom(T,-): D^b(X) \longrightarrow D^b(\text{mod-} B)
$$
will restrict to give an equivalence of full subcategories 
$$
    \RHom(T,-): D^b_0(X) \longrightarrow D^b_0(\text{mod-} B).
$$
\end{lemma}
\begin{proof}
    We only need to show that given an object $F\in D_0^b(X)$, $\BR \Hom(T,-)$ sends to $D^b_0(\text{mod-} B)$. We consider the spectral sequence
    $$
        E^{p,q}_{2} = \Ext^q_{X}\left(T,H^p(F)\right)\ \Rightarrow \ \Ext^{p+q}_{X}(T,F).
    $$
    Each $E^{p,q}_2$ is finite dimensional since $H^p(F)$ is supported on $\CP^2$, thus the spectral sequence converges to terms which are all finite dimensional, thus nilpotent over $B$.
\end{proof}
The standard t-structure on $D^b(\text{mod-} B)$ induces one on $D^b_0(\text{mod-}B)$ in the obvious way, and we pull this back using the equivalence
$$
    \RHom_X(T,-): D^b_0(X) \longrightarrow D^b_0(\text{mod-} B)
$$
of the above lemma gives a bounded t-structure on $D^b_0(X)$. If $T = \bigoplus \pi^* E_i$ for some exceptional collection on $\CP^2$ in Theorem \ref{tilt_obj},  we denote the induced heart by $\CB(E_0,E_1,E_2)$. Then $\CB(E_0,E_1,E_2)$ is a finite length abelian category with $3$ simple objects, which we can write down explicitly: denote the dual exceptional collection of $\BE$ by $\BF = (F_0,F_1,F_2)$ as in Definition-Lemma \ref{dual_coll}. Then the simple objects in $\CB(E_0,E_1,E_2)$ will be 
\begin{equation}\label{simple_exc}
S_0 :=s_*F_0, \quad S_1:=s_*F_1, \quad S_2:=s_* F_2.     
\end{equation}
where we denote the inclusion map $s: \CP^2\rightarrow X$.

 Here we introduce a conception which encode the information of simple tilts:
\begin{definition}[Ext-quivers]
Let $\HH$ be a  t-structure of finite length in triangulated category $D$, then the Ext-quiver of $\HH$ is the graded quiver $Q(\HH)$ whose vertexes are labeled by simple objects $S_i$, and whose graded $k$ arrows $S_i\rightarrow S_j$ corresponds to a basis  of $\Hom_D(S_j, S_i[k])$.
\end{definition}
We denote the heart of t-structure by $\CA$ which is obtained from the tilting object $T = \pi^* \CO(-2)\oplus \pi^* \CO(-1)\oplus \pi^* \CO $, that is $\CA =\CB(\CO(-2),\CO(-1),\CO)$. 

For $\CA$ we have
\begin{enumerate}
    \item the simple objects which are $S_0 = s_* \CO(-2), \ S_1 = s_* \Omega(-1)[1], \ S_2= s_* \CO(-3)[2]$;
    \item the dual objects in $D^b(X)$ which are $P_0 = \pi^*\CO(-2),\ P_1 = \pi^*\CO(-1),\ P_2 = \pi^*\CO$ such that $\Hom^{\bullet}_{D^b(X)} ( P_i, S_j) = \delta_{ij} \BC$;
    \item the arrows in the Ext-quiver which will be: 
        \begin{eqnarray*}
            \Ext^i_X(S_0,S_1) &=& \Ext^i_X\left(s_*\CO(-2), s_*\Omega(-1)[1]\right) \\
                            &=& \Ext^{i+1}_{\CP^2}\left(s^*s_*\CO(-2),\Omega(-1)\right) \\ 
                            &=& \bigoplus_{k=0}^2\Ext^{i+1-k}_{\CP^2}\left(\CO(-2)\otimes \wedge^k\TT, \Omega(-1)\right) \\
                            &=& \bigoplus_{k=0}^2\HO^{i+1-k}\left(\CP^2, \Omega(1)\otimes \wedge^k\Omega\right),
        \end{eqnarray*}
        in the third equation we use the formula obtained from the Koszul resolution $s^*s_* \CO_{\CP^2} = \bigoplus \wedge^k \TT[k] $ in $D^b(\CP^2)$ \cite[Proposition 11.1]{Huy06}. We firstly tensor the vector bundle $\Omega(1)$ by the exterior powers of the Euler exact sequence:
        $$
            \xymatrix{
                0 \ar[r] & \wedge^k \Omega \ar[r] & \wedge^k V(-1) \ar[r] & \wedge^{k-1} \Omega \ar[r] & 0
            }
        $$
        Then we take the long exact sequence of cohomology groups, and do the inductions on $k$ of $\wedge^k \Omega$. Finally we get
        $$
            \HO^p(\CP^2, \Omega(1)\otimes \wedge^q \Omega) = \left\{ \begin{array}{cc}
                                                                        \BC& p=1 \text{ and } q=1,\ 2\\
                                                                        \BC^3 & p=2 \text{ and } q = 2 \\
                                                                        0 & \text{ otherwise.}
                                                                        \end{array}\right.
        $$
        
        Thus we have
        $$
            \Ext^i_X(S_0,S_1) = \left\{\begin{array}{cc}
                                        \BC^3 & i = 1,\ 3\\
                                        0 & \text{otherwise}
                                        \end{array}
                                \right.
        $$
        Similarly we have
        \begin{eqnarray*}
            \Ext^i_X(S_1,S_2) &=& \left\{\begin{array}{cc}
                                        \BC^3 & i = 1,\ 3\\
                                        0 & \text{otherwise}
                                        \end{array}
                                \right. \\
            \Ext^i_X(S_0,S_2) &=& \left\{\begin{array}{cc}
                                        \BC^3 & i = 2\\
                                        0 & \text{otherwise}
                                        \end{array}
                                \right. \\   
           (m=0,\ 2)\  \Ext^i_X(S_m,S_m) &=& \left\{\begin{array}{cc}
                                        \BC & i = 0,\ 2,\ 4\\
                                        0 & \text{otherwise}
                                        \end{array}
                                \right. \\ 
                \Ext^i_X(S_1,S_1) &=& \left\{\begin{array}{cc}
                                        \BC & i=0,\ 4\\
                                        \BC^{10} & i = 2\\
                                        0 & \text{otherwise}
                                        \end{array}
                                \right.
        \end{eqnarray*}
        the other arrows could be obtained from the Serre duality\\ $\Hom_D(S_i, S_j[k]) \cong \Hom_D(S_j,S_i[4-k])^*$.
    \item We draw the Ext-quiver of $\CA$ as the following: the black arrows are of degree $1$, the red arrows are of degree $2$:
        \begin{equation}\label{Ext-A}
            \begin{tikzcd}
                S_0 \arrow[out=120,in=180,loop,red,"1"]   \arrow[r,bend left, "3"] &  S_1 \arrow[out=60,in=120,loop,red,"10"] \arrow[l,bend left, "3"] \arrow[r,bend left,"3"] & S_2 \arrow[out=0,in=60,loop,red,"1"] \arrow[l,bend left, "3"]
                \arrow[curve={height=-60pt},red,"3", from=1-1, to=1-3] \arrow[curve={height=-60pt},red,"3", from=1-3, to=1-1] 
           \end{tikzcd}
        \end{equation}
         we omit the arrows of other degree since they can be read from the Serre duality.   
\end{enumerate}
\subsection{Mukai flops}
According to the Ext-quiver (\ref{Ext-A}) of $\CA$,  we  see that there is a symmetry between $S_0$ and $S_2$. In fact, we can write down explicitly the auto-equivalence of $D^b(X)$ which interchanges  $S_0$ and $S_2$. This is quite useful for simplifying the calculations of the simple tilts. 
\begin{definition}[Mukai flop]
Let $g: Z\rightarrow X$ a blow-up along the zero section of the projection $\pi: X\rightarrow \CP^2$. The exceptional locus $E(\subset Z)$ of $g$ is the incidence variety in $\CP^2\times (\CP^2)^{\vee}$ where $(\CP^2)^{\vee}$ is the dual projective space. By blowing down $Z$ along the projection $E\rightarrow (\CP^2)^{\vee}$, we obtain a birational map $g^+:Z\rightarrow X^+ = \Tot \Omega_{(\CP^2)^{\vee}}$. The resulting birational map
$$
   \xymatrix{ \phi = g^+\circ g^{-1}:X \ar@{..>}[r] & X^+} 
$$
is called a Mukai flop. 
\end{definition}
The following result were proved in \cite[Theorem 3.1]{Nam03} \cite[Corollary 5.7]{Kaw02}:
\begin{theorem}[Derived equivalence for Mukai flop]\label{Mukai}
Let $\xymatrix{ \phi:X \ar@{..>}[r] & X^+}$ denote the Mukai flop between $X$ and $X^+$, then we have a derived equivalence $\Phi: D^b(X) \rightarrow D^b(X^+)$. Moreover, the derived equivalence can be restricted to $D^b_0(X) \rightarrow D^b_0(X^+)$.
\end{theorem}
\begin{proof}
Here we adopt the proof from \cite[Example 5.3]{TodUeh}: put $\CO_X(1) := \pi^*\CO(1)$. Then we have a tilting object $\varepsilon = \oplus_{i=0}^2 \CO_X(-i)$ by Theorem \ref{tilt_obj}. Since $\phi$ is an isomorphism in codimension one, there is an equivalence between the categories of reflexive sheaves on $X$ and $X^+$. Denote $\varepsilon'$ the sheaf corresponding to $\varepsilon$ on $X^+$. By results in \cite[Lemma 1.3]{Nam03}, we have 
$$
    \varepsilon' \cong \bigoplus_{i=0}^2\CO_{X^+}(i).
$$
Since we have the isomorphism of rings
$$
    \xymatrix{
        \phi_*:\End_{X}(\varepsilon) = A  \ar[r]^-{\cong} & A' = \End_{X^+}(\varepsilon')
    }
$$
so we have the equivalence of derived categoires:
$$
    \xymatrix{
        \Phi: D^b(X') \ar[r] & D^b(A') \ar[r] & D^b(A) \ar[r] & D^b(X).
    }
$$
By restricting to full subcategories, we have an equivalence:
$$
    \xymatrix{
        \uline{\Phi}: D^b_0(X') \ar[r] & D^b_0(A') \ar[r] & D^b_0(A) \ar[r] & D^b_0(X).
    }
$$
\end{proof}
Then under the canonical isomorphism between $\CP^2$ and $(\CP^2)^{\vee}$, we can view the $\Phi$ as an autoequivalence of $D^b(X)$ (resp. $D^b_0(X)$).
\begin{proposition}\label{swap}
Let $\BE = \left(\CO(-2),\CO(-1), \CO\right)$,  $\CA$ be the associated bounded t-structure on $D^b_0(X)$, and $S_i$ be the simple objects in $\CA$. Let $\Psi := \left(-\otimes^{\BL} \CO_X(-2)\right)\circ \Phi$, then we have $ \Psi(\CA) = \CA$, and $\Psi(S_j) = S_{2-j}$ for $j=0,\ 1,\ 2$.
\end{proposition}
\begin{proof}
Since by Theorem \ref{Mukai}, $\pi^*\CO(i)$ ($-2\leq i\leq 0$) is sent to $\pi^*\CO(-i)$ via $\Phi$, then it follows that $\Psi\left(\pi^*\CO(i)\right) = \pi^*\CO(-i-2)$ for $-2\leq i \leq 0$. Thus we have $\Psi\left(\bigoplus_{i=-2}^0 \pi^* \CO(i)\right) = \bigoplus_{i=-2}^0 \pi^* \CO(i)$, and we get that $\Psi(\CA) \cong \CA$. 

If we denote by $P_0 = \pi^*\CO(-2)$, $P_1 =\pi^*\CO(-1)$ and $P_2 = \pi^*\CO$, then the $\Psi(S_j)$ is uniquely determined by the dual relation with $\Psi\left(P_i\right)$, i.e 
$$
    \Hom_{D^b(X)}^{\bullet}\left(\Psi(P_i), \Psi(S_j)\right) =             \left\{\begin{array}{cc}
                \BC & -i=j\\
                0 & \text{otherwise}
                \end{array}
    \right.
$$
Since $\Psi\left(P_i\right) = P_{2-i}$, thus we have $\Phi(S_j) = S_{2-j}$ for $j = 0,\ 1,\ 2$. 
\end{proof}

We proceed to calculate the simple tilts at $S_i$ of $\CA$, and their Ext-quivers:
\subsection{Simple tilt at $S_1$}
        Denote by $\mathcal{C} := L_{S_1} \CA$, we list the following information:
        \begin{enumerate}
            \item by Proposition \ref{sim_tilt} the simple objects  are $U_1  = S_1[-1] = s_*\Omega(-1), \ U_0,\ U_2 $ where $U_0$ and $U_2$  fit into the following short exact sequences:
             \begin{eqnarray}
                     0 \rightarrow s_*\CO(-2) \rightarrow U_0  \rightarrow s_*\Omega(-1)[1]^3 = S_1\otimes \Ext^1(S_1,S_0) \rightarrow 0
                  \label{univ_seq1}\\
                 0 \rightarrow s_*\CO(-3)[2] \rightarrow U_2 \rightarrow s_*\Omega(-1)[1]^3 = S_1\otimes \Ext^1(S_1,S_2) \rightarrow 0
                \label{univ_seq2} 
             \end{eqnarray}
             note that by Proposition \ref{swap} we have $\Psi(U_0) = U_2$ and $\Psi(U_2) = U_0$.
             
            We need the following lemma for our calculations:
            \begin{lemma}
                 $U_0 = s_* P[1]$ where $P\in \Coh \CP^2$ fits into the following short exact sequence:
                    \begin{equation} \label{PP}
                        \xymatrix{
                        0\ar[r] & P \ar[r] & \Omega(-1)\otimes \Hom_{\CP^2}\left(\Omega(-1), \CO(-2) \right) \ar[r] & \CO(-2) \ar[r]& 0
                        }
                    \end{equation}
                    that is, $P = L_{\Omega(-1)} \CO(-2)$ (the left mutation).
             \end{lemma}
             \begin{proof}
                 Since $s$ is closed immersion, then $s_*$ is exact. Then we push-forward the above short exact sequence by $s_*$, then shift the resulting short exact sequence by degree $1$. Note that 
                    \begin{eqnarray*}
                        \Ext^1(S_1,S_0) &=& \Ext^1_X(s_*\Omega(-1)[1], s_*\CO(-2)) \\
                                        &=& \Hom_X(s_* \Omega(-1), s_*\CO(-2)) \\
                                        &=& \Hom_{\CP^2}(\Omega(-1), \CO(-2)),
                    \end{eqnarray*}
                    we have $U_0 \cong s_* P[1]$ by the uniqueness of the cone. 
             \end{proof}
             Moreover, $P$ is again an exceptional object   and $(\CO(-3), \Omega(-1), P)$ forms a full and strong exceptional collection on $\CP^2$ by Theorem 2.10.

             \item The dual objects in $D^b(X)$ are $\wtil{T_0},\ \pi^*\CO(-2),\ \pi^*\CO$, where $\wtil{T_0}$ fits into the universal extension sequence 
             $$
                  \xymatrix{
                      0 \ar[r] & \pi^*\Omega(-1) \ar[r] & \wtil{T_0} \ar[r] & \pi^* \CO^3 \ar[r] & 0.
                }
             $$

            \item 
            \begin{itemize}
                \item 
                 To calculate $\Ext^i_X(U_1,U_0)$, we take the long exact sequence of (\ref{univ_seq1}) associated with $\Hom_X(U_1,-) = \Hom_X(S_1[-1],-)$:
                  \begin{equation*} 
                       \begin{tikzcd}
                           \Ext^1_X(U_1,S_0) = 0\arrow[r] 
                                &  \Ext^1_X(U_1,U_0) \arrow[r] \arrow[d,phantom,""{coordinate, name=Z}] 
                                  & \Ext^1_X(U_1,S_1^{\oplus 3}) = \BC^{30}  \arrow[dll,
                                                            "\delta",
                                                            rounded corners,
                                                            to path={ -- ([xshift=2ex]\tikztostart.east)
                                                        |- (Z) [near end]\tikztonodes
                                                        -| ([xshift=-2ex]\tikztotarget.west)
                                                        -- (\tikztotarget)}] \\
                             \Ext^2_X(U_1,S_0) = \BC^3 \arrow[r] 
                                & \Ext^2_X(U_1,U_0) \arrow[r]  \arrow[d,phantom,""{coordinate, name=W}] 
                                      &  \Ext^2_X(U_1,S_1^{\oplus 3}) = 0  \arrow[dll,
                                                            "",
                                                            rounded corners,
                                                            to path={ -- ([xshift=2ex]\tikztostart.east)
                                                        |- (W) [near end]\tikztonodes
                                                        -| ([xshift=-2ex]\tikztotarget.west)
                                                        -- (\tikztotarget)}] \\
                           \Ext^3_X(U_1,S_0) = 0\arrow[r] 
                                  & \Ext^3_X(U_1,U_0) \arrow[r]  \arrow[d,phantom,""{coordinate, name=W}] 
                                   &  \Ext^3_X(U_1,S_1^{\oplus 3}) = \BC^3  \arrow[dll,
                                                                 "",
                                                            rounded corners,
                                                            to path={ -- ([xshift=2ex]\tikztostart.east)
                                                        |- (W) [near end]\tikztonodes
                                                        -| ([xshift=-2ex]\tikztotarget.west)
                                                        -- (\tikztotarget)}]\\
                             \Ext^4_X(U_1,S_0)=0\arrow[r]  
                                & \Ext^4_X(U_1,U_0) \arrow[r] 
                                    & \Ext^4_X(U_1,S_1^{\oplus 3}) = 0
                        \end{tikzcd}
                         \end{equation*}
                To determine the surjectivity of the connection map $\delta$, we can show that \\$\Ext^2_X(U_1,U_0) = 0$: by Serre duality, we have
                    \begin{eqnarray*}
                        \Ext^2_X(U_1,U_0) &=& \Ext^2_X(s_*\Omega(-1), s_* P[1]) \\
                        &=& \Ext^1_X(s_*P, s_*\Omega(-1) )^*
                    \end{eqnarray*}
                    Then we use the Koszul resolution again, we have
                    \begin{eqnarray*}
                        \Ext^1_X(s_*P, s_*\Omega(-1)) &=& \Ext^1_{\CP^2}(s^*s_* P, \Omega(-1)) \\
                        &=& \Ext^1_{\CP^2}(P, \Omega(-1)) \oplus \Hom_{\CP^2}(P\otimes \TT, \Omega(-1))
                    \end{eqnarray*}
                    The first component vanishes since $P$ and $\Omega(-1)$ forms part of  a strong exceptional collection. By (\ref{PP}) we get that $c_0(P) = 5$ and $c_1(P) = -13$, thus $c_0 (P\otimes \TT) = 10$ and $c_1(P\otimes \TT) = -11$. So we have $\mu(P\otimes \TT) = -10/11 > \mu(\Omega(-1)) = -5/2$. Since $P$, $\TT$ and $\Omega(-1)$ are slope semistable sheaves, and so is $P\otimes \TT$ (see \cite[Theorem 3.1.4]{HL10}), there is no map between the semistable sheaves from the larger slope to the smaller one, so $\Hom_{\CP^2}(P\otimes \TT,\Omega(-1)) = 0$. We have shown that $\Ext^2_X(U_1,U_0) = 0$, by taking this to the above long exact sequence, we can read that
            \begin{eqnarray*}
                 \Ext^i_X(U_1,U_0) &=& \left\{\begin{array}{cc}
                                        \BC^{27} & i = 1\\
                                        \BC^3 & i=3 \\
                                        0 & \text{otherwise}
                                        \end{array}
                                \right.
             \end{eqnarray*}
                \item For $\Ext^i_X(U_1,U_2)$, we apply the functor $\Psi$, and using that $\Psi(U_1) = U_1$, $\Psi(U_0) = U_2$, we have
                         \begin{eqnarray*}
                              \Ext^i_X(U_1,U_2) =  \Ext^i_X(U_1,U_0) =  \left\{\begin{array}{cc}
                                        \BC^{27} & i = 1\\
                                        \BC^3 & i=3 \\
                                        0 & \text{otherwise}
                                        \end{array}
                                \right.
                     \end{eqnarray*}
             \item 
                To get $\Ext^i(U_2,U_0)$, we first  take the long exact sequence of (\ref{univ_seq2}) associated with $\Hom(-,U_0)$ (note that $\Ext^i_X(S_1,U_0) = \Ext^{i-1}_X(U_1,U_0)$):
                \begin{equation}\label{ses_U2U0} 
                       \begin{tikzcd}
                           \Ext^1_X(S_1,U_0)^3 = 0\arrow[r] 
                                &  \Ext^1_X(U_2,U_0) \arrow[r] \arrow[d,phantom,""{coordinate, name=Z}] 
                                  & \Ext^1_X(S_2,U_0) \arrow[dll,
                                                            "",
                                                            rounded corners,
                                                            to path={ -- ([xshift=2ex]\tikztostart.east)
                                                        |- (Z) [near end]\tikztonodes
                                                        -| ([xshift=-2ex]\tikztotarget.west)
                                                        -- (\tikztotarget)}] \\
                             \Ext^2_X(S_1,U_0)^3 = \BC^{81} \arrow[r] 
                                & \Ext^2_X(U_2,U_0) \arrow[r]  \arrow[d,phantom,""{coordinate, name=W}] 
                                      &  \Ext^2_X(S_2,U_0)  \arrow[dll,
                                                            "",
                                                            rounded corners,
                                                            to path={ -- ([xshift=2ex]\tikztostart.east)
                                                        |- (W) [near end]\tikztonodes
                                                        -| ([xshift=-2ex]\tikztotarget.west)
                                                        -- (\tikztotarget)}] \\
                           \Ext^3_X(S_1,U_0)^3 = 0\arrow[r] 
                                  & \Ext^3_X(U_2,U_0) \arrow[r]  \arrow[d,phantom,""{coordinate, name=W}] 
                                   &  \Ext^3_X(S_2,U_0) \arrow[dll,
                                                                 "",
                                                            rounded corners,
                                                            to path={ -- ([xshift=2ex]\tikztostart.east)
                                                        |- (W) [near end]\tikztonodes
                                                        -| ([xshift=-2ex]\tikztotarget.west)
                                                        -- (\tikztotarget)}]\\
                             \Ext^4_X(S_1,U_0)^3=\BC^9\arrow[r]  
                                & \Ext^4_X(U_2,U_0)=0 \arrow[r] 
                                    & \Ext^4_X(S_2,U_0)
                        \end{tikzcd}
                         \end{equation}
            We first determine 
            $$\Ext^i_X(S_2,U_0) = \Ext^i_X(s_*\CO(-3)[2], s_*P[1]) = \Ext^{i-1}_X(s_*\CO(-3),s_*P),
            $$ where $P$ was defined in the short exact sequence (\ref{PP}). 
            \begin{enumerate}
                \item $i=1$, then 
                    \begin{eqnarray*}
                        \Ext^1_X(S_2, U_0) &=&\Hom_X(s_*\CO(-3), s_* P) \\
                                        &=& \Hom_{\CP^2}(\CO(-3),P)
                    \end{eqnarray*}
                    then we take $\Hom(\CO(-3),-)$ to (\ref{PP}), and note that $\Ext^1(\CO(-3),P) = 0$ since $(\CO(-3),\Omega(-1),P)$ forms a strong exceptional collection, then we get that 
                    $$
                        \Hom_{\CP^2}(\CO(-3),P) = \BC^6.
                    $$
                    Thus $\Ext^1_X(S_2,U_0) = \BC^6$.
                \item $i=2$, then
                    \begin{eqnarray*}
                       \quad \Ext^2_X(S_2,U_0) &=& \Ext^1_{\CP^2}(\CO(-3), P) \oplus \Hom_{\CP^2}(\CO(-3)\otimes \TT, P) 
                    \end{eqnarray*}
                    then first component vanishes. We have  $c_0\left(\TT(-3)\right) = 2$, $c_1\left(\TT(-3)\right) = 0$, then $\mu(\TT(-3)) = 0 > \mu(P) = -13/5$, so there is no map from $\TT(-3)$ to $P$ since they are both semistable sheaves. Thus
                    $$
                        \Ext^2_X(S_2,U_0) = 0.
                    $$
                \item $i=3$, then 
                        \begin{eqnarray*}
                         \quad   \Ext^3_X(S_2,U_0) &=& \Ext^2_{\CP^2}(\CO(-3), P) \oplus \Ext^1_{\CP^2}(\CO(-3)\otimes \TT, P) \\
                                        & & \oplus \Hom_{\CP^2}(\CO, P)
                        \end{eqnarray*}
                        the first component vanishes, the third component vanishes because $\mu(\CO) > \mu(P)$ and they are semistable sheaves. For the second component, we take $\Hom_{\CP^2}(\TT(-3),-)$ on short exact sequence (\ref{PP})
                        \begin{equation}\label{ses_TT}
                        \resizebox{6in}{!}{
                        \xymatrix{
                             \Hom_{\CP^2}\left(\TT(-3), \CO(-2)\right) \ar[r] & \Ext^1_{\CP^2}\left(\TT(-3),P\right) \ar[r]&
                             \Ext^1_{\CP^2}\left(\TT(-3),\Omega(-1)\right)^3 \ar[r] &
                             \Ext^1_{\CP^2} \left(\TT(-3),\CO(-2)\right)  
                        }
                        }
                        \end{equation}
                        
                        We have $\Hom_{\CP^2}(\TT(-3),\CO(-2)) = \Ext^1_{\CP^2}(\TT(-3),\CO(-2)) = 0$ by the Bott formula. By taking $\Hom_{\CP^2}(\TT(-2),-)$ on the Euler sequence, 
                         we have \\ $\Ext^1_{\CP^2}(\TT(-2),\Omega) = \Ext^1_{\CP^2}\left(\TT(-3),\Omega(-1)\right) = \BC^3$. Back to (\ref{ses_TT}), we have $\Ext^1_{\CP^2}(\TT(-3),P) = \BC^9$. So finally we have 
                        $$
                        \Ext^3_X(S_2,U_0) = \BC^9.
                        $$
            \end{enumerate}
            Taking the above infomation back to (\ref{ses_U2U0}), we get that $ \Ext^3_X(U_2,U_0) = 0$ by the exactness. By Serre duality and applying the autoequivalence $\Psi$, we can also get $\Ext^1_X(U_2,U_0)$:
            \begin{eqnarray*}
                \Ext^1_X(U_2,U_0) &= & \Ext^3_X(U_0,U_2)^* \\
                                &=& \Ext^3_X\left(\Psi(U_0),\Psi(U_2)\right)^* \\
                                &=& \Ext^3_X(U_2,U_0)^*\\
                                &=& 0.
            \end{eqnarray*}
            To determine $\Ext^2_X(U_2,U_0)$, we can calculate the Euler characteristic 
            $$\chi(U_2,U_0) = \sum_{i=0}^4 (-1)^i\dim \Ext^i_X(U_2,U_0)$$
            This can be easily done by reducing to the level of Grothendieck group: we denote $[A]$ the equivalent class in $K_0(X)$, then 
            \begin{equation}
                [U_i] = 3[S_1]+[S_i]
            \end{equation}
            in $K_0(X)$ for $i = 0,\ 2$ by definition. Then 
            \begin{eqnarray*}
                \chi\left([U_2],[U_0]\right) &=& \chi\left(3[S_1]+[S_2], 3[S_1]+[S_0]\right) \\
                                &=& 9\chi\left([S_1],[S_1]\right) + 3\chi\left([S_1],[S_0]\right) + 3\chi\left([S_2],[S_1]\right) + \chi\left([S_2],[S_0]\right) 
            \end{eqnarray*}
            From the calculation of the Ext-quiver of $\CA$, we know that $\chi\left(S_1,S_1\right) = 12$, \\ $\chi\left([S_1],[S_0]\right) = \chi\left([S_2],[S_1]\right) = -6$, and $\chi\left([S_2],[S_0]\right) = 3$. So we obtain that 
            \begin{eqnarray*}
                \chi(U_2,U_0) &=& \sum_i (-1)^i\dim \Ext^i_X(U_2,U_0) \\
                             &=& \dim \Ext^2_X(U_2,U_0)\\
                             &=& 75
            \end{eqnarray*}
            So we have $\Ext^2_X(U_2,U_0) = \BC^{75}$. In summary, we have
            \begin{eqnarray*}
              \Ext^i_X(U_2,U_0) &=& \left\{\begin{array}{cc}
                                        \BC^{75} & i=2 \\
                                        0 & \text{otherwise}
                                        \end{array}
                                \right.
             \end{eqnarray*}
            And the $\Ext^i_X(U_0,U_2) $ is obtained by the Serre duality:
            \begin{eqnarray*}
              \Ext^i_X(U_0,U_2) &=& \Ext^{4-i}_X(U_2,U_0)^*\\
                                &=&
                                    \left\{\begin{array}{cc}
                                        \BC^{75} & i=2 \\
                                        0 & \text{otherwise}
                                        \end{array}
                                \right.
             \end{eqnarray*}
        \item
            For $\Ext^i_X(U_0, U_0)$, we have $\Hom_X(U_0,U_0) = \Ext^4_X(U_0,U_0)^* = \BC$, and 
            \begin{eqnarray*}
                \Ext^1_X(U_0,U_0) &=& \Ext^1_X(s_*P,s_*P) \\
                                 &=& \Ext^1_{\CP^2}(P,P) \oplus \Hom_{\CP^2}(P\otimes \TT, P).
            \end{eqnarray*}
            Then the first component vanishes because $P$ is exceptional. In the calculations of $\Ext^i(U_1,U_0)$, we have $\mu(P\otimes T) = -10/11 > \mu(P) = -13/5$. Then the second component also vanishes. 
            
            Also we obtain  $\Ext^3_X(U_0,U_0) = \Ext^1_X(U_0,U_0)^* = 0$.
            
            Then we calculate the Euler characteristic $\chi(U_0,U_0)$:
            \begin{eqnarray*}
                \chi(U_0,U_0) &=& \chi\left(3[S_1]+[S_0],3[S_1]+[S_0]\right)\\
                &=& 9 \chi(S_1,S_1) + 3\chi(S_1,S_0) + 3\chi(S_0,S_1)+ \chi(S_0,S_0) \\
                &=& 75
            \end{eqnarray*}
            We have $\chi(U_0,U_0) = 2+\dim \Ext^2_X(U_0,U_0) = 75$, so $\Ext^2_X(U_0,U_0) = \BC^{73}$. In summary, we have

            \begin{eqnarray*}
            \Ext^i_X(U_0,U_0) &=& \left\{\begin{array}{cc}
                                        \BC^1 & i = 0,\ 4 \\
                                        \BC^{73} & i=2 \\
                                        0 & \text{otherwise}
                                        \end{array}
                                \right.
            \end{eqnarray*}
         By applying the functor $\Psi$, we have 
                    \begin{eqnarray*}
                        \Ext^i_X(U_2,U_2) &=& \Ext^i_X\left(\Psi(U_0),\Psi(U_0)\right) \\
                                        &=&
                                            \left\{\begin{array}{cc}
                                        \BC^1 & i = 0,\ 4 \\
                                        \BC^{73} & i=2 \\
                                        0 & \text{otherwise}
                                        \end{array}
                                \right.
            \end{eqnarray*}
        \item
        Finally for $\Ext^i_X(U_1,U_1)$ we have 
        \begin{eqnarray*}
            \Ext^i_X(U_1,U_1) &=& \Ext^i_X(S_1,S_1)\\
                              &=& \left\{\begin{array}{cc}
                                        \BC & i=0,\ 4\\
                                        \BC^{10} & i = 2\\
                                        0 & \text{otherwise}
                                        \end{array}
                                \right.
        \end{eqnarray*}
        \end{itemize}
        
        \item the Ext quiver $Q(\CC)$ will be:
          \begin{equation} \label{ext_qui_C}
                \begin{tikzcd} 
                U_0 \arrow[out=120,in=180,loop,red,"73"]  \arrow[r,bend left, "27"] & U_1 \arrow[out=60,in=120,loop,red,"10"] \arrow[l,bend left,"3"] \arrow[r,bend right,"3"] & U_2 \arrow[out=0,in=60,loop,red,"73"]  \arrow[l,bend right,"27"] \arrow[curve={height=-50pt},red,"75", from=1-3, to=1-1] \arrow[curve={height=-50pt},red,"75", from=1-1, to=1-3] 
                \end{tikzcd}
        \end{equation}
        as before the black arrows are of degree 1, the red arrows are of degree 2, and we omit the arrows of other degree since they can be read from the Serre duality.
\end{enumerate} 

\subsection{Simple tilt at $S_2$}
We denote by $\Gamma := L_{S_2}\CA$ and $\Gamma' := L_{S_0} \CA$. Combining the  Proposition \ref{swap} and Lemma \ref{aut_sim_tilt}  we have 
\begin{eqnarray*}   
    \Psi\left(L_{S_0}\CA\right) \cong L_{\Psi(S_0)}\Psi\left(\CA\right)  \cong L_{S_2} \CA.
\end{eqnarray*}
Thus $\Gamma' = L_{S_0}\CA \cong \Gamma = L_{S_2}\CA$. Under the equivalence we only need to calculate one of the Ext-quivers (here we choose $\Gamma$):
    \begin{enumerate}
        \item the simple objects are $W_0 = s_*\CO(-2),\ W_1 = s_*\CO(-3)[1], \ W_2 = s_* \CO(-1)[2]$;
        \item the dual objects  in $D^b(X)$ are $\pi^* \CO(-2),\ \pi^*\Omega,\  \pi^*\CO(-1)$;
        \item 
            \begin{itemize}
                \item calculations of $\Ext^i_X(W_1,W_0)$:
            \begin{eqnarray*}
                \Ext^i_{X}\left(s_* \CO(-3)[1], s_*\CO(-2) \right) &=& \bigoplus_{k=0}^2\Ext^{i-k-1}_{\CP^2}\left(\CO(-3)\otimes \wedge^k\TT, \CO(-2)\right) \\    &=&\bigoplus_{k=0}^2 \HO^{i-k-1}\left(\CP^2, \Omega^k \otimes \CO(1)\right) \\
                &=& \left\{\begin{array}{cc}
                     \BC^3 & i = 1 \\
                     0 & \text{otherwise}
                \end{array}
                \right.
            \end{eqnarray*}
        \item calculations of $\Ext^i_X(W_2,W_0)$: 
            \begin{eqnarray*}
                \Ext^i_X\left(s_* \CO(-1)[2], s_*\CO(-2) \right) &=& \bigoplus_{k=0}^2\Ext^{i-k-2}_{\CP^2}\left(\CO(-1)\otimes \wedge^k\TT, \CO(-2)\right) \\    &=&\bigoplus_{k=0}^2 \HO^{i-k-2}\left(\CP^2, \Omega^k \otimes \CO(-1)\right) \\
                &=& \left\{\begin{array}{cc}
                     \BC^6 & i = 2 \\
                     \BC^3 & i = 3 \\
                     0 & \text{otherwise}
                \end{array}
                \right.
            \end{eqnarray*}
        \item similarly for $\Ext^i_X(W_2,W_1)$, we have
                \begin{eqnarray*}
                \Ext^i_X\left(s_* \CO(-1)[2], s_*\CO(-3)[1] \right)   &=&\bigoplus_{k=0}^2 \HO^{i-k-1}\left(\CP^2, \Omega^k \otimes \CO(-2)\right) \\
                &=& \left\{\begin{array}{cc}
                     \BC^3 & i = 1 \\
                     0 & \text{otherwise}
                \end{array}
                \right.
            \end{eqnarray*}
        \item for $\Ext^i_X(W_j,W_j)$ ($j=0,\ 2$) we have
            \begin{eqnarray*}
                \Ext^i_{X}\left(W_j,W_j \right)   &=&\bigoplus_{k=0}^2 \HO^{i-k}\left(\CP^2, \Omega^k \right) \\
                &=& \left\{\begin{array}{cc}
                     \BC & i = 0, \ 2,\ 4 \\
                     0 & \text{otherwise}
                \end{array}
                \right.
            \end{eqnarray*}
        
            \end{itemize}
        \item the Ext-quiver $Q(\CB)$ will be 
         $$
            \begin{tikzcd}
                     W_0 \arrow[out=120,in=180,loop,red,"1"]  \arrow[r,bend left, "3"] & W_1 \arrow[out=60,in=120,loop,red,"1"]  \arrow[r,bend left,"3"] & W_2 \arrow[out=0,in=60,loop,red,"1"]  \arrow[curve={height=-50pt},red,"6", from=1-1, to=1-3] \arrow[curve={height=-50pt},red,"6", from=1-3, to=1-1] 
                     \arrow[curve={height=-30pt},"3", from=1-3, to=1-1] 
                    \end{tikzcd}
                    $$
    \end{enumerate}

\bibliographystyle{plain}
\bibliography{biblio}
\end{document}